\documentclass[english,a4paper,11pt]{amsart}

\usepackage[T1]{fontenc} 
\usepackage[latin1]{inputenc}
\usepackage{amssymb, babel}
\usepackage{amssymb}
\usepackage{amsmath}

\usepackage{amssymb}
\usepackage{amsmath}

\newtheorem{theorem}{Theorem}[section] 
\newtheorem{lemma}[theorem]{Lemma} 
\newtheorem{corollary}[theorem]{Corollary} 
\newtheorem{proposition}[theorem]{Proposition}
 
\theoremstyle{definition}
\newtheorem*{ack}{Acknowledgments}
\newtheorem{definition}[theorem]{Definition} 
\newtheorem{assump}[theorem]{Assumptions}
\newtheorem{remark}[theorem]{Remark}

\title[]{Lie conformal superalgebras and duality of modules over linearly compact
Lie superalgebras}
\DeclareMathOperator{\str}{str}
\DeclareMathOperator{\ad}{ad}
\DeclareMathOperator{\diver}{div}
\DeclareMathOperator{\tr}{tr}
\author{Nicoletta Cantarini}\author{Fabrizio Caselli}\author{Victor Kac}

\begin{document}
\maketitle
\thispagestyle{empty}
\newcommand{\inlinewedge}{\textrm{\raisebox{0.6mm}{\footnotesize $\bigwedge$}}}
\newcommand{\displaywedge}{\textrm{\raisebox{0.6mm}{\tiny $\bigwedge$}}}

\def\C{{\mathbb F}}
\def\R{{\mathbb R}}
\def\Z{{\mathbb Z}}
\def\la{\textrm{{\boldmath $\lambda$}}}
\def\deb{\textrm{{\boldmath $\de$}}}
\def\mub{\textrm{{\boldmath $\mu$}}}
\def\nub{\textrm{{\boldmath $\nu$}}}
\def\yb{\textrm{{\boldmath $y$}}}
\def\xb{\textrm{{\boldmath $x$}}}
\def\zb{\textrm{{\boldmath $z$}}}
\def\wb{\textrm{{\boldmath $w$}}}
\def\xib{\textrm{{\boldmath $\xi$ \hspace{-1mm}}}}
\def\de{\partial}
\def\sl{\mathfrak{sl}}

\begin{abstract}
We construct a duality functor on the category of continuous representations of linearly compact Lie superalgebras, using representation theory of Lie conformal superalgebras. We compute the dual representations of the generalized Verma modules.
\end{abstract}
\section{Introduction}
Lie conformal superalgebras encode the singular part of the operator product expansion (OPE) of
chiral fields in the vacuum sector of conformal field theory \cite{K1}:
\begin{equation}
a(z)b(w)=\sum_{j\in\Z}\frac{a(w)_{(j)}b(w)}{(z-w)^{j+1}}.
\label{I1}
\end{equation}
They play an important role in the theory of vertex algebras that encode the full OPE
\eqref{I1}, so that the full structure of a vertex algebra is captured by the $\lambda$-bracket of the Lie
conformal (super)algebra structure:
\begin{equation}
[a(w)_\lambda b(w)]=\sum_{j\geq 0}(a(w)_{(j)}b(w))\frac{\lambda^j}{j!},
\label{I2}
\end{equation}
and the normally ordered product $a(w)_{(-1)}b(w)$ (since $a(w)_{(-n-1)}b(w)=
\frac{1}{n!}(\de_w^na(w)_{(-1)}b(w))$). In the classical limit the normally ordered product of a vertex algebra becomes commutative, but the $\lambda$-bracket still satisfies the axioms of a Lie conformal superalgebra. This leads to the theory of Poisson vertex algebras that play a fundamental role in the theory of Hamiltonian PDE.

Recall that $\lambda$-bracket \eqref{I2} satisfies the following axioms, where $a=a(w)$,
$\de a=\de_wa(w)$:
\begin{itemize}
\item[] (sesquilinearity)\,\,\, $[\de a_{\lambda}b]=-\lambda[a_\lambda b]$, $[a_{\lambda}\de b]=(\de +\lambda)[a_\lambda b]$,
\item[]  (skewsymmetry)\,\, $[b_{\lambda}a]=-(-1)^{p(a)p(b)}[a_{-\lambda-\de} b]$,
\item[] (Jacobi identity)\, $[a_{\lambda}[b_\mu c]]=[[a_\lambda b]_{\lambda+\mu}c]+(-1)^{p(a)p(b)}[b_\mu[a_\lambda c]]$.
\end{itemize}
An $\C[\de]$-module $R$, endowed with a $\lambda$-bracket $R\otimes R\rightarrow R[\lambda]$,
satisfying the above three axioms, is called a {\it Lie conformal superalgebra}. Here and in the sequel we denote by $\C$ an algebraically closed field of characteristic 0.

All the work on representation theory of Lie conformal superalgebras $R$ was based on the simple observation that representations of $R$ are closely related to ``continuous'' representations
of the associated to $R$ annihilation Lie superalgebra. Recall that the annihilation Lie
superalgebra is the vector superspace
\begin{equation}
\mathcal A(R)=R[[t]]/(\de+\de_t)R[[t]]
\label{I3}
\end{equation}
where $t$ has even parity, with the (well-defined) continuous bracket
\begin{equation}
[at^m,bt^n]=\sum_{j\geq 0}{{m}\choose{j}}(a_{(j)}b)t^{m+n-j},
\label{I4}
\end{equation}
which makes it a linearly compact Lie superalgebra.
Since $\de$ commutes with $\de_t$, it extends in a natural way to a derivation $\de$ of the Lie
superalgebra $\mathcal A(R)$, hence one can define the extended annihilation Lie superalgebra $\mathcal A^e(R)=\C[\de]\ltimes
\mathcal A(R)$. It is easy to see that a ``conformal'' $R$-module $M$ is the same as a continuous  
$\mathcal A^e(R)$-module \cite{CK1}.
In most of the examples the derivation $\de$ is an inner derivation of $\mathcal A(R)$: $\de=\ad a$, 
$a\in \mathcal A(R)$, so that $\mathcal A(R)=\mathcal A^e(R)/(\de-\ad\, a)\mathcal A^e(R)$.
An $R$-module ($=\mathcal A^e(R)$-module) $M$ is called {\it coherent} if $(\de-a)M=0$. Thus a continuous
$\mathcal A(R)$-module is the same as a coherent $R$-module.

Note that if a Lie conformal superalgebra $R$ is a finitely generated $\C[\de]$-module, then $\mathcal A(R)$ is a linearly compact Lie superalgebra, hence representation theory of Lie conformal superalgebras is intimately related to 
the theory of continuous representations of linearly compact Lie superalgebras.

Although it is unclear what is a right definition of a vertex algebra in several indeterminates,
the definition of a Lie conformal superalgebra and all the above remarks can be easily extended to the
case when one even indeterminate $\lambda$ is replaced by several even commuting indeterminates
$\lambda_1, \dots, \lambda_r$. In the paper we allow also for $s$ odd indeterminates $\lambda_{r+1}, \dots, \lambda_{r+s}$ and we say that the corresponding
Lie conformal superalgebra is of type $(r,s)$, but for the sake of simplicity this will not be discussed in the introduction. In the present paper we reverse the point of view: instead of using continuous representations of linearly compact Lie superalgebras in the study
of finitely generated $\C[\de]$-modules over Lie conformal superalgebras, we use the latter
to study the former. But then a natural question arises: which linearly compact Lie superalgebras are annihilation algebras of Lie conformal superalgebras? The answer probably is: in all interesting examples. Namely, if
a linearly compact Lie superalgebra is of geometric origin, i.e., it is constructed with the use of vector 
fields and differential forms in a formal neighborhood of a point in an $(r|s)$-dimensional
supermanifold, then it is an annihilation superalgebra of a finitely generated as an $\C[\de_1, \dots, \de_r]$-module
Lie conformal superalgebra in $r$ indeterminates.

Let us demonstrate this on the example of the Lie algebra $W(r)$ of continuous derivations of the algebra of formal power series $\C[[t_1,\dots, t_r]]$. Include the Lie algebra $W(r)$ in a larger
Lie algebra $\widetilde{W}(r)$ of continuous derivations of the algebra of formal Laurent series 
$\C[[t_1,\dots, t_r]][t_1^{-1}, \dots, t_r^{-1}]$. Consider the $\widetilde{W}(r)$-valued formal distributions
$$a_i(\zb)=-\sum_{k_1,\dots, k_r\in\Z^r}(t_1^{k_1}\dots t_r^{k_r}\de_{t_i})z_1^{-k_1-1}\dots z_r^{-k_r-1}.$$
It is immediate to see, using the standard properties of the formal
$\delta$-function $\delta(z-w)=\sum_{n\in\Z}z^nw^{n-1}$, that
$$[a_i(\zb), a_j(\wb)]=\de_{w_i}a_j(\wb)\delta(\zb-\wb)+a_i(\wb)\de_{w_j}\delta(\zb-\wb)+
a_j(\wb)\de_{w_i}\delta(\zb-\wb),$$
where $\zb=(z_1,\dots, z_r)$, $\wb=(w_1,\dots, w_r)$, and  $\delta(\zb-\wb)=\prod_{i=1}^r\delta(z_i-w_i)$.
Applying the formal Fourier transform, i.e., letting
$$[a_\la b]=Res_{\zb}[a(\zb), b(\wb)]e^{\sum_i\lambda_iz_i},$$
we obtain a structure of a Lie conformal algebra, which we denote by $RW(r)$, on the free
$\C[\de_{z_1}, \dots, \de_{z_r}]$-module of rank $r$, generated by the elements $a_i=a_i(\zb)$, with the
following $\la$-brackets:
$$[{a_i}_\la a_j]=\de_ia_j+\lambda_ja_i+\lambda_ia_j,~~i,j=1,\dots,r.$$
It is easy to see that the linearly compact Lie algebra $W(r)$ is the annihilation algebra of the Lie conformal algebra $RW(r)$.

A remarkable feature of representation theory of a Lie conformal superalgebra $R$ is the existence of
a contravariant {\it duality functor} on the category of $R$-modules which are finitely generated
as $\C[\de]$-modules \cite{K1, BDAK, BKLR}. Extension of this construction to the case of a Lie conformal 
superalgebra $R$ in several indeterminates is straightforward. Due to the above remarks,
this duality functor can be transported to the category of continuous representations of the linearly compact Lie superalgebra, which is the annihilation algebra of $R$.

In the present paper we study the duality functor for the category $\mathcal P$ of continuous
$\Z$-graded modules with discrete topology over $\Z$-graded linearly compact Lie superalgebras
${\mathfrak g}=\prod_{j\in\Z_{\geq -d}}{\mathfrak g}_j$, where the depth $d\geq 1$, $\dim {\mathfrak g}_j<\infty$,
and $[{\mathfrak g}_i,{\mathfrak g}_j]\subset {\mathfrak g}_{i+j}$. We have: ${\mathfrak g}={\mathfrak g}_{<0}+{\mathfrak g}_{\geq 0}$ where 
${\mathfrak g}_{<0}=\bigoplus_{j<0}{\mathfrak g}_j$ and ${\mathfrak g}_{\geq 0}=\prod_{j\geq 0}{\mathfrak g}_j$ . We also require that modules from $\mathcal P$ are finitely generated as
${\mathcal U}({\mathfrak g}_-)$-modules. Recall that a $\mathfrak g$-module $M$ is continuous if, for any $v\in M$, $\mathfrak g_n v=0$ for $n$ 
sufficiently large.

The category $\mathcal P$ is similar to the BGG category $\mathcal O$, and, as in category
$\mathcal O$, the most important objects in $\mathcal P$ are {\it{generalized Verma modules}} $M(F)$
(see, e.g.\ \cite{KR1}). Recall, that, given a finite-dimensional ${\mathfrak g}_0$-module $F$, one extends it trivially
to the subalgebra ${\mathfrak g}_{>0}=\prod_{j> 0}{\mathfrak g}_j\subset {\mathfrak g}$, 
and defines 
$$M(F)=\textrm{Ind}_{{\mathfrak g}_{\geq 0}}^{{\mathfrak g}} F.$$
Our main result is the computation of the dual to $M(F)$ ${\mathfrak g}$-module $M(F)^*$ (see Theorems 
\ref{dualmodule78} and \ref{main}). It turns out that $M(F)^*$ is not $M(F^*)$, but $M(F^\vee)$,
where $F^\vee$ is a shifted ${\mathfrak g}_0$-module $F^*$ by the following character (=1-dimensional representation) $\chi$ 
of ${\mathfrak g}_0$:
\begin{equation}
\label{beautifulShift}
\chi(a)=\str(\ad a|_{\mathfrak{g}_{<0}}).
\end{equation}
 In particular, if ${\mathfrak g}_0=[{\mathfrak g}_0,{\mathfrak g}_0]$, then
$\chi=0$, and $M(F)^*=M(F^*)$, which happens, for example,
for the principally graded exceptional Lie superalgebra $E(5,10)$ \cite{K}. This observation has been
made in \cite{CC}, which led us to the present paper.

One of the main problems of representation theory of a linearly compact Lie superalgebra $\mathfrak g$ is the classification
of {\it degenerate} (i.e.\ non-irreducible) generalized Verma modules $M(F)$, associated to
finite-dimensional irreducible ${\mathfrak g}_0$-modules $F$. Since the topological dual of $M(F)$ endowed with discrete
topology is a linearly compact $\mathfrak g$-module, a solution of the above problem is important for the description of irreducible linearly compact $\mathfrak g$-modules.

In order to apply these results to representation theory of simple, finite rank Lie conformal superalgebras of type
$(r,s)$, one needs to develop representation theory of the corresponding annihilation algebra,
which apart from the ``current'' case, is a central extension of an infinite-dimensional simple linearly compact Lie superalgebra.

The simple infinite-dimensional linearly compact Lie superalgebras were classified and explicitly described, along with their maximal open subalgebras, in \cite{K,CK2, CaK}. It was shown in 
\cite{CK} and \cite{FK} that all linearly compact simple Lie superalgebras of growth 1 (rather their universal central extensions) are annihilation superalgebras of simple Lie conformal superalgebras of type $(1,0)$.
Using them, all finite rank simple Lie conformal superalgebras of type $(1,0)$ were classified in \cite{FK}. A complete
list of those, admitting a non-trivial $\Z$-gradation, consists of three series $W_N$, $S_N$, and $K_N$, and two exceptions: 
$K'_4$ and $CK_6$.

Representation theory of $W_N$ and $S_N$ was constructed in \cite{BKLR}, of $K_N$ with $N=0,1$ in \cite{CK1},
resp.\ with $N=2,3,4$ in \cite{CL}, resp.\ with all $N\geq 0$ in \cite{BKL}.
Representation theory of $CK_6$ and $K'_4$ was constructed in \cite{BKL1} and \cite{B}, respectively. A very interesting feature
of these works is that all degenerate modules are members of infinite complexes, for classical (resp.\ exceptional)
Lie conformal superalgebras the number of these complexes being one or two (resp.\ infinite).

A complete representation theory of linearly compact Lie superalgebras, corresponding to 
simple Lie conformal superalgebras of type $(r,s)$ with $r>1$ is known only for Cartan type Lie algebras, beginning with the
paper \cite{R0}, and for the exceptional Lie superalgebra $E(3,6)$ \cite{KR1, KR2, KR3}. Some partial results
in other cases are obtained in \cite{KR} and \cite{R}. Note that again for all known examples the degenerate modules
can be organized in infinite complexes, the number of them being finite (resp.\ infinite) in the classical
(resp.\ exceptional) cases. We hope that the duality, established in the present paper, and the Lie conformal superalgebra approach will help to make
progress in representation theory in the remaining cases, especially $E(5,10)$.

The contents of the paper are as follows. After the introduction we discuss in Section \ref{section2} the notion of Lie
conformal superalgebra of type $(r,s)$, its annihilation Lie superalgebra, and elements of their representation theory.
In Section \ref{two} we introduce the duality functor in the category of finitely-generated modules over a Lie
conformal superalgebra of type $(r,s)$ and of the corresponding annihilation Lie superalgebra. We prove here the main Theorem
\ref{main} under Assumptions \ref{assumption}. We conjecture that
if $\mathfrak{g}$ is a 
linearly compact Lie superalgebra then for any transitive pair $(\mathfrak{g}, \mathfrak{g}_{\geq 0})$, i.e., such that $\mathfrak{g}_{\geq 0}$ is an open subalgebra of
$\mathfrak g$ containing no non-trivial ideal of $\mathfrak g$,  one can construct a duality functor for which Theorem \ref{main} still holds with
$\chi(a)=\str(\ad a|_{\mathfrak{g}/\mathfrak{g}_{\geq 0}})$ for $a\in \mathfrak{g}_{\geq 0}$.
In the remaining Sections \ref{section4}--\ref{section8} we show that the linearly compact Lie superalgebras
$\mathfrak{g}=W(r,s)$, $K(1,n)$, $E(5,10)$, $E(3,6)$, and $E(3,8)$ are annihilation Lie superalgebras
of certain Lie conformal superalgebras $R\mathfrak{g}$ of type $(r,s)$ (for suitable $r$ and $s$) which we describe explicitely. We check
that in all these cases Assumptions \ref{assumption} on $\mathfrak{g}$ with its principal gradation are satisfied. We also check in Section \ref{section4} that for all
annihilation superalgebras, associated to the ordinary Lie conformal superalgebras (i.e., of type $(1,0)$)
Theorem \ref{main} is applicable as well. Unfortunately we do not know whether this is the case for the remaining
exceptional simple Lie superalgebra $\mathfrak{g}=E(4,4)$, though it is not difficult to construct the corresponding $R\mathfrak{g}$. 
\begin{ack} The first two authors were partially supported by PRIN 2015: ``Moduli spaces and Lie theory''. The third author was partially supported by the Bert and Ann Kostant fund, and the Simons fellowship. 
\end{ack}
\section{Lie conformal superalgebras of type $(r,s)$}\label{section2}
Let $\Z_+=\{0,1,2,\ldots\}$. Fix once and for all two non-negative integers $r$ and $s$.
We will use several sets of $r+s$ variables such as $\lambda_1,\ldots,\lambda_{r+s}$, $\de_1,\ldots,\de_{r+s}$, $y_1,\ldots,y_{r+s}$. We will always assume that variables with indices $1,\ldots,r$ are even and variables with indices $r+1,\ldots,r+s$ are odd, and accordingly we let $p_i=0$ if $i=1,\ldots,r$ and $p_i=1$ if $i=r+1,\ldots,r+s$. We will also use bold letters such as $\la$ or $\deb$ or $\yb$ to denote the set of corresponding variables.
We denote by $\inlinewedge[\la]=\C[\lambda_1,\ldots,\lambda_r]\otimes \inlinewedge(\lambda_{r+1},\ldots,\lambda_{r+s})$ and we similarly define $\inlinewedge[\deb]$ or $\inlinewedge[\yb]$. The completion
$\inlinewedge[[\yb]]$ of $\inlinewedge[\yb]$ is the algebra of formal power series in $\yb$.

If $R$ is a $\mathbb Z/2\mathbb{Z}$-graded vector space we give to $\inlinewedge [\la]\otimes R$ the structure of a $\mathbb Z/2\mathbb{Z}$-graded $\inlinewedge [\la]$-bimodule by letting $\lambda_i(P(\la)\otimes a)=\lambda_iP(\la)\otimes a$ and $(P(\la)\otimes a)\lambda_i=(-1)^{p_ip(a)} P(\la)\lambda_i\otimes a$, where $p(a)\in \mathbb Z/2\mathbb{Z}$ denotes the parity of $a$. We will usually drop the tensor product symbol and simply write $P(\la)a$ instead of $P(\la)\otimes a$.

\begin{definition}\label{defsupconf}
	A {\it{Lie conformal superalgebra of type}} $(r,s)$ is a $\mathbb Z/2\Z$ graded $\inlinewedge[\deb]$-bimodule $R$ such that $a\de_i=(-1)^{p_ip(a)}\de_ia$ for all $a\in R$ and $i\in \{1,\ldots,r+s\}$, endowed with a $\la$-bracket, i.e. a $\mathbb Z/2\Z$-graded linear map $R\otimes R\rightarrow \inlinewedge[\la]\otimes R$, denoted by $a\otimes b\mapsto [a_\la b]$, that satisfies the following properties:
	
		\begin{align}& [(\de_i a)_\la b]=-\lambda_i[a_\la b];& \label{defsupconf1}\\
		&[a_\la (b\de_i)])=[a_\la b](\de_i+\lambda_i);&\label{defsupconf2}\\
		&[b_\la a]=-(-1)^{p(a)p(b)}[a_{-\la-\deb}b];&\label{defsupconf3}\\
		 &[a_\la [b_\mub c]] = [[a_\la b]_{\la+\mub}c]+(-1)^{p(a)p(b)}[b_\mub[a_\la c]].&\label{defsupconf4}
		 \end{align}
	 	
\end{definition}
We refer to Properties \eqref{defsupconf1} and \eqref{defsupconf2}  as  the conformal sesquilinearity, Property \eqref{defsupconf3} as the conformal skew-symmetry and
Property \eqref{defsupconf4} as the conformal Jacobi identity.

We note that the notion of a Lie conformal superalgebra, as treated in \cite{K1}, corresponds to a
Lie conformal superalgebra of type $(1,0)$. 
For the convenience of the reader we first briefly present the theory of Lie conformal superalgebras of type $(r,0)$: 
in this case all the results are straightforward extensions of those in type $(1,0)$ and therefore are stated
without proofs. We then develop the general theory in type $(r,s)$.

If $K=(k_1,\ldots,k_r)$ is any $r$-tuple of non-negative integers we let
$$\la^K=\lambda_1^{k_1}\lambda_2^{k_2}\cdots \lambda_r^{k_r} \,\,{\mbox{and}}\,\,
K!=k_1!\cdots k_r!.$$

For $a,b\in R$ the $K$-products $(a_Kb)$ are defined by the polynomial expansion  
\begin{equation}
[a_\la b]=\sum_{K\in \Z_{+}^r}\frac {\la^K}{K!}(a_Kb).
\label{5a}\end{equation}

Starting from a Lie conformal superalgebra $R$ of type $(r,0)$, one can construct a new
Lie conformal superalgebra $\tilde R$ of the same type, called the affinization of $R$. Let $\tilde R=R\otimes\C[[\yb]]$. We consider $\tilde R$ as a $\C[\deb_\yb]=\C[\de_{y_1},\ldots,\de_{y_{r}}]$-module, with $\la$-bracket given by
\begin{equation}\label{lambdatildepari}
[(a\yb^M)_\la (b\yb^N)]=\big([a_{\la+\deb_{\yb}}b]\yb^M\big)\yb^N.
\end{equation}
Note that in this expression it is meant that the derivatives with respect to the variables $y_i$
in the bracket $[a_{\la+\deb_{\yb}}b]$ 
act only on $\yb^M$.
The corresponding $K$-products are:
\begin{equation}\label{Kproductspari}
((a\yb^M)_K(b\yb\,\, ^N))=\sum_{J\in\Z_+^r}\frac{1}{J!}{(a_{K+J}b)((\deb _{\yb}})^J\yb ^M)\yb ^N.
\end{equation}
The $\la$-bracket \eqref{lambdatildepari} defines on $\tilde R$ a Lie conformal superalgebra structure with $\tilde \deb= \deb+\deb_\yb$.

	 The annihilation algebra associated with $R$ is the Lie superalgebra
	 \[
	  \mathcal A(R)=\tilde R /\tilde \deb \tilde R
	 \]
	 with the bracket given by
	 \[
	  [a\yb^M, b\yb^N]=\sum_{J}\frac{1}{J!}{(a_{J}b)((\deb _{\yb}})^J\yb ^M)\yb^N=[({a\yb^M})_{\lambda}(b\yb^N)]_{|\la=\bf 0}.
	 \]

The  representation theory of a Lie conformal superalgebra is closely related to the representation theory of the corresponding annihilation algebra. This fact relies on the following relation
\begin{equation}\label{fundeven}[a_{\la},b_{\mub}]=[a_{\la}b]_{\la+\mub}, ~~a,b\in R,
\end{equation}
where 
$$a_{\la}=\sum_{K\in \Z_+^r}\frac{\la^{K}}{K!} a\yb^K \in \mathcal{A}(R)[[\la]].$$

The goal of this section is to extend these results to Lie conformal superalgebras of type $(r,s)$.
In this context, in order to simplify computations involving signs, it is more convenient to use
expansion \eqref{7a} below instead of \eqref{5a}. For this, introduce 
the following notation. 
If $K=(k_1,\ldots,k_t)$ is any sequence with entries in $\{1,\ldots,r+s\}$ we let
\[m_i(K)=|\{j\in\{1,\ldots,t\}:\, k_j=i\}|,
\]

\[f(K)=\prod_i m_i(K)!,
\]
and
\[\la_K=\lambda_{k_1}\lambda_{k_2}\cdots \lambda_{k_t}
\]
and we similarly define $\yb_K$, $\xb_K$ and so on. 
If $K=\emptyset$, we let $f(K)=1$, $\lambda_K=1$.
We also let $p_K=p_{k_1}+\cdots +p_{k_t}$ and so $p(\la_K)=p_K$.
For example, if $r=2$, $s=3$ and $K=(2,3,2,1,5,4)$, then $f(K)=2$, $\la_K=-\lambda_1\lambda_2^2\lambda_3\lambda_4\lambda_5$ and $p_K=1$.
It is clear that if $J$ and $K$ are obtained from each other by a permutation of the entries we have $\la_J=\pm \la_K$; we write in this case $J\sim K$ and we denote by $S_{r,s}$ any set of representatives of these equivalent classes.  
For $a,b\in R$ the $K$-products $(a_Kb)$ are uniquely determined by  the following conditions:
\begin{itemize}
\item[-] $(a_Kb)=0$ if $\la_K=0$ (i.e. if $K$ contains a repeated odd index);
\item[-] $\la_J(a_Jb)=\la_K (a_Kb)$ for $J\sim K$;
\end{itemize}
\begin{equation} 
[a_\la b]=\sum_{K\in S_{r,s}}\frac {\la_K}{f(K)}(a_Kb).\label{7a}
\end{equation}

\begin{remark} Conditions \eqref{defsupconf1} and \eqref{defsupconf2} in Definition \ref{defsupconf} can be restated in terms of $K$-products by means of the following equations, where, for $K=(k_1,\ldots,k_t)$
and $i\in\{1,\dots, r+s\}$, we let $iK=(i,k_1,\ldots,k_t)$:
	\begin{align}
		 &((\de_ia)_Kb)=0 \textrm{ if }m_i(K)=0 \textrm{ and }((\de_i a)_{iK}b)=-(m_i(K)+1)(a_Kb)
		\,{};& \label{Kprod1}\\
		 &(a_K(b\de_i))=(a_Kb)\de_i\textrm{ if }m_i(K)=0 \textrm{ and }  \label{Kprod2} \\&(a_{iK}(b\de_i))=(a_{iK}b)\de_i+(m_i(K)+1)(-1)^{(p(a)+p(b))p_i}(a_Kb) \,{}.& \nonumber
	\end{align}
\end{remark}
As in the completely even case, starting from a Lie conformal superalgebra $R$ of type $(r,s)$ one can construct a new Lie conformal superalgebra $\tilde R$ of the same type, called the {\it affinization} of $R$. Let 
$\tilde R=R\otimes \inlinewedge[[\yb]]$. We consider $\tilde R$ as a $\inlinewedge[[\yb]]$-bimodule and also as a $\inlinewedge[\deb_\yb]=\inlinewedge[\de_{y_1},\ldots,\de_{y_{r+s}}]$-bimodule letting
\[
 \de_{y_i} a\yb_M=(-1)^{p_ip(a)}a(\de_{y_i} \yb_M)=(-1)^{p_i(p(a)+p_M)}a\yb_M \de_{y_i}
\]
with $\la$-bracket given by
\begin{equation}\label{lambdatilde}
[(\yb_M a)_\la (b \yb_N)]=\big(\yb_M[a_{\la+\deb_{\yb}}b]\big)\yb_N.
\end{equation}
The corresponding $K$-products are:
\begin{equation}\label{Kproducts}
((\yb_M a)_K(b\yb _N))=(-1)^{p_Kp_M}\sum_{J\in S_{r,s}}\frac{1}{f(J)}(\yb _M{(\deb _{\yb}})_J)(a_{KJ}b)\yb _N,
\end{equation}
where for $K=(k_1,\ldots,k_t)$ and $J=(j_1,\ldots,j_u)$ we let $KJ=(k_1,\ldots,k_t,j_1,\ldots,j_u)$.
\begin{proposition} The $\inlinewedge(\tilde{\deb})$-module $\tilde R$ with $\tilde \deb= \deb+\deb_\yb $ and $\la$-bracket given by \eqref{lambdatilde} is a Lie conformal superalgebra.	
\end{proposition}
\begin{proof}We first check condition \eqref{defsupconf1} in Definition \ref{defsupconf}. 
	\begin{align*}
	[(\tilde \de_i\yb _Ma)_\la b\yb _N]&=[((\de_i+\de_{y_i})\yb_M a)_\la b\yb _N]\\
	&= (-1)^{p_ip_M}[(\yb_M \de_i a)_\la b\yb _N]+(\de_{y_i}\yb _M[a_{\la+\deb_{\yb }}b])\yb_N\\
	&= (-1)^{p_ip_M}(\yb _M[\de_i a_{\la+\deb_\yb }b])\yb _N+(\de_{y_i}\yb _M[a_{\la+\deb_\yb }b])\yb_N \\
	&=-(-1)^{p_ip_M}(\yb _M(\lambda_i+\de_{y_i})[a_{\la+\deb_\yb }b])\yb_N+(\de_{y_i}\yb _M[a_{\la+\deb_\yb }b])\yb_N\\
	&= -\lambda_i (\yb_M[a_{\la+\deb_\yb} b])\yb_N\\
	&=-\lambda_i [(\yb_M a)_{\la} (b \yb_N)]
	\end{align*}
	Similarly one can check condition \eqref{defsupconf2}.
	Now we verify the conformal skew-symmetry, i.e.
	\[
	[(\yb _M a) _\la (b\yb _N)]=-(-1)^{(p(a)+p_M)(p(b)+p_N)}[(b\yb_N)_{-\la-\deb-\deb_\yb }(\yb_Ma)].\]
	We have
	\begin{align*}
	[(\yb _M a) _\la (b\yb _N)]&=(\yb_M[a_{\la+\deb_{\yb}}b])\yb_N\\
	&=-(-1)^{p(a)p(b)}\big(\yb_M[b_{-\la-\deb_{\yb}-\deb}a]\big)\yb_N\\
	&= -(-1)^{p(a)p(b)+(p(a)+p(b))(p_M+p_N)+p_Mp_N}\yb_N \big([b_{-\la-\deb_\yb -\deb}a]\yb_M\big)\\
	&= -(-1)^{p(a)p(b)+(p(a)+p(b))(p_M+p_N)+p_Mp_N}[(\yb_N b)_{-\la-\deb-\deb_\yb }(a \yb_M)]
	\end{align*}
	where the last equality holds due to the Leibniz rule.
	The verification of the conformal Jacobi identity is left to the reader.	
\end{proof}
\begin{definition}\label{annihilationalgebra}
	Given a Lie conformal superalgebra $R$ of type $(r,s)$, the {\it{annihilation Lie superalgebra}} associated to $R$ is the  vector super space
	\[
	\mathcal A(R)=\tilde R /\tilde \deb \tilde R,
	\]
	with bracket given by
	\[
	[\yb _M a,b\yb _N]=\sum_{J\in S_{r,s}}\frac{1}{f(J)}(\yb _M{(\deb _{\yb}})_J)(a_{J}b)\yb _N
	=[(\yb_M a)_\la (b\yb_N)]_{|\la=0}.
	\]
	
\end{definition}
\begin{proposition}
	$\mathcal A(R)$ is a Lie superalgebra.
\end{proposition}
\begin{proof}
	The proof is a straightforward generalization of the standard conformal case which is treated in \cite{K1}. One can also check that this is an immediate consequence of Properties \eqref{defsupconf3} and \eqref{defsupconf4} on $\tilde R$ together with the observation that $\de_i+\de_{y_i}=0$ on $\mathcal A(R)$.
\end{proof}
Next target is to extend the fundamental identity $\eqref{fundeven}$ to Lie conformal superalgebras. The crucial point here is to give the appropriate definition of $a_\la \in \inlinewedge[[\la]]\otimes \mathcal A(R)$: this is the main result in the next Proposition \ref{lemmino}. We first give some technical lemmas.
\begin{lemma}\label{lem1}
	Let $R,J$ be two finite sequences with entries in $\{1,\ldots,r+s\}$. Then
	\[
	\frac{1}{f(RJ)}\yb_{RJ}(\deb_\yb)_J=\frac{\eta_J}{f(R)}\yb_R,
	\]
	where $\eta_J=(-1)^{1+2+\cdots+q_J}$ and $q_J$ is the number of odd entries in $J$.
\end{lemma}
\begin{proof}
	It easily follows by induction on the length of $J$.
\end{proof}
For $K=(k_1,\ldots,k_t)$ a sequence with entries in $\{1,\ldots,r+s\}$ we let $\bar K=(k_t,\ldots,k_1)$.
If $K$ is the empty set then $\bar K=K$.
\begin{lemma}\label{2.7}
	Let $K$ be a finite sequence with entries in $\{1,\ldots,r+s\}$. We have
	\[
	\frac{1}{f(K)}(\la+\mub)_K \yb_K= \sum_{I,R\in S_{r,s}:\, IR\sim K }\frac{1}{f(I)f(R)}\la_I\mub_R \yb_I \yb_R.
	\]
	and
	\[
	\frac{1}{f(K)}(\la+\mub)_{\bar K} \yb_K= \sum_{I,R\in S_{r,s}:\, IR\sim K }\frac{1}{f(I)f(R)}\la_{\bar R}\mub_{\bar I} \yb_I \yb_R.
	\]
\end{lemma}
\begin{proof}
	It is sufficient to notice that
	\[
	\frac{1}{f(K)}(\la+\mub)_K = \sum_{I,R\in S_{r,s}:\, IR\sim K} \varepsilon_{I,R;K}\frac{1}{f(I)f(R)}\la_I\mub_R.
	\]
	where, if $IR\sim K$,  $\varepsilon_{I,R;K}$ is defined by $\la_I\la_R=\varepsilon_{I,R;K}\la_K$ and therefore $\varepsilon_{I,R;K}=\varepsilon_{\bar R,\bar I;\bar K}$.
\end{proof}

As in the completely even case the following result will turn out to be crucial in the representation theory of Lie conformal superalgebras (cf.\ \cite{CK1}).

\begin{proposition}\label{lemmino}
	Let $R$ be a Lie conformal superalgebra of type $(r,s)$. For $a\in R$, let
	$$a_{\la}=\sum_{K\in S_{r,s}}(-1)^{p_K}\frac{\la_{\bar K}}{f(K)} \yb_K a \in \mathcal{A}(R)[[\la]].$$
	Then $$[a_{\la},b_{\mub}]=[a_{\la}b]_{\la+\mub}.$$
\end{proposition}
\begin{proof}
	We have
	\begin{align*}
	[a_\la,b_\mub]&=\sum_{K,H}(-1)^{p_K}\frac{\la_{\bar K}}{f(K)}[\yb_Ka,b\yb_H]\frac{\mub_{\bar H}}{f(H)}\\
	&=\sum_{K,H,J}(-1)^{p_K}\frac{\la_{\bar K}}{f(K)f(J)} (\yb_K(\deb_\yb)_J)(a_Jb)\yb_H   \frac{\mub_{\bar H}}{f(H)}\\
	&=\sum_{R,H,J}(-1)^{p_R+p_J}\frac{\la_{\bar J}\la_{\bar R}}{f(J)f(RJ)}(\yb_{RJ}(\deb_\yb)_J)(a_Jb)\yb_H   \frac{\mub_{\bar H}}{f(H)}\\
	&= \sum_{R,H,J}(-1)^{p_R+p_J}\frac{\la_{\bar J}\la_{\bar R}}{f(J)f(R)}\eta_J \yb_{R}(a_Jb)\yb_H   \frac{\mub_{\bar H}}{f(H)}
	\end{align*}
	by Lemma \ref{lem1}. Now notice that $\lambda_J=(-1)^{p_J}\eta_J\lambda_{\bar J}$ hence
	\begin{align*}
	[a_\la,b_\mub]&=\sum_{R,H}(-1)^{p_R}\frac{\la_{\bar R}}{f(R)f(H)}\yb_{R}\yb_H \mub_{\bar H}\sum_J\frac{\la_J}{f(J)}(a_Jb)\\
	&= \sum_{R,H}(-1)^{p_R+p_H}\frac{\la_{\bar R}}{f(R)f(H)}\mub_{\bar H}\yb_{H}\yb_R \sum_J\frac{\la_J}{f(J)}(a_Jb)\\
	&=\sum_S (-1)^{p_S}\frac{(\la+\mub)_{\bar S}}{f(S)}\yb_S \sum_J\frac{\la_J}{f(J)}(a_Jb)\\
	&=[a_\la b]_{\la+\mub},
	\end{align*}
	by Lemma \ref{2.7}.
	
\end{proof}
\section{Conformal modules and conformal duals}\label{two}
This section is dedicated to the study of modules over Lie conformal superalgebras.
\begin{definition}
	A  {\it{conformal module}} $M$ over a Lie conformal superalgebra $R$ of type $(r,s)$ is a ${\mathbb Z}/2\Z$-graded 
	$\inlinewedge[\deb]$-module with a ${\mathbb Z}/2\Z$-graded linear map
	$$R\otimes M\rightarrow \displaywedge[\la]\otimes M, ~~
	a\otimes v\mapsto a_{\la}v$$
	such that
	\begin{itemize}
		\item[(M1)] $(\de_i a)_{\la}v=[\de_i,a_{\la}]v=-\lambda_ia_{\la}v$;
		\item[(M2)] $[a_{\la},b_{\mub}]v=a_{\la}(b_{\mub}v)-(-1)^{p(a)p(b)}b_{\mub}(a_{\la})v=
		(a_{\la}b)_{\la+\mub}v.$
	\end{itemize} 
\end{definition}

\begin{definition} A Lie conformal superalgebra $R$ of type $(r,s)$ is called $\Z$-graded if $\inlinewedge[\la]\otimes R=\oplus_{d\in\Z} (\inlinewedge[\la]\otimes R)_d$, where $(\inlinewedge[\la]\otimes R)_d$ denotes the
	homogeneous component of degree $d$, and for every homogeneous elements $a,b\in \inlinewedge[\la]\otimes R$ one has:
	\begin{itemize}
		\item[i)] $\deg(\lambda_i a)=\deg (a)-2$;
		\item[ii)] $\deg(\partial_i a)=\deg (a)-2$;
		\item[iii)] $\deg[a_\la b]=\deg(a)+\deg(b)$.
	\end{itemize}
\end{definition}

Notice that if $R$ is a $\Z$-graded Lie conformal superalgebra of type $(r,s)$ then its annihilation algebra $\mathcal A(R)$ inherits a
$\Z$-gradation by setting \[\deg(a\yb_M)=\deg(a)+2\,\ell(M),\] where, for $M=(m_1,\dots,m_t)$, $\ell(M)=t$.

In what follows we assume the following technical conditions on a Lie conformal superalgebra $R$ of type $(r,s)$, which turn out to be satisfied in many interesting cases: we state them explicitly for future reference. 
\begin{assump}\label{assumption} ${}$
	\begin{enumerate}
		\item $R$ is $\Z$-graded;
		\item The induced $\Z$-gradation on $\mathcal A(R)$ has depth at most 3;
		\item The homogeneous components  $\mathcal A(R)_{-1}$ and  $\mathcal A(R)_{-3}$ are purely odd;
\item The map $\ad:\mathcal A(R)_{-2}\rightarrow \textrm{der} (\mathcal A(R))$ is injective and its image is  $D=\langle \partial_{y_1},\dots, \partial_{y_{r+s}}\rangle$. 
\end{enumerate}
\end{assump}
We consider the semi-direct sum of Lie superalgebras $D\ltimes  \mathcal A(R)$, we observe that the subset $I=\{x-\ad(x),\, x\in \mathcal A(R)_{-2}\}$ is an ideal and we set
\[
 \mathfrak g(R)=(D \ltimes \mathcal A(R))/I.
\]
By Assumptions \ref{assumption} (4), $\mathfrak{g}(R)\cong\mathcal A(R)$ and 
$\mathfrak{g}(R)_{-2}\cong D$. Indeed we will identify $\mathfrak g(R)_{-2}$ with $D=\langle \de_{y_1},\ldots,\de_{y_{r+s}}\rangle$.

We point out that the Lie superalgebra $D\ltimes \mathcal A(R)$ is the natural generalization of the so-called extended annihilation algebra introduced by Cheng and Kac in \cite{CK1}. A key observation made in \cite{CK1} is that conformal $R$-modules are exactly the same as continuous (called conformal in \cite{BKLR}) modules over the extended annihilation algebra. The following proposition extends this result in our context and is proved using Proposition \ref{lemmino}.

\begin{proposition}
\label{modcorr}
A conformal $R$-module is precisely a continuous module over $D\ltimes \mathcal A(R)$, i.e. a module $M$ such that for every $v\in M$ and every $a\in R$, $(\yb_Ka).v\neq 0$ only for a finite number of $K$. 
The equivalence between the two structures is provided by the following relations:
\begin{itemize}
\item $a_\la v=\sum_{K}(-1)^{p_K}\frac{\la_{\bar K}}{f(K)} (\yb_Ka).v$;
\item $\de_i v=-\de_{y_i}.v$.
\end{itemize}
\end{proposition}


\begin{definition} \label{defcoherent}We say that a conformal $R$-module $M$ is coherent if the action of the ideal $I=\{x-\ad(x),\, x\in \mathcal A(R)_{-2}\}$ is trivial.
\end{definition}
Thanks to Proposition \ref{modcorr}, a coherent $R$-module is precisely a continuous module over $\mathfrak g(R)$.
One checks directly that a conformal $R$-module $M$ is coherent if and only if it satisfies the following property:
for every $a\in R$ and all $K$ such that $\deg (\yb_K a)=-2$, i.e.\ ${\yb_K a}=\sum_i\alpha_i \partial_{y_{i}}\in
\mathfrak g(R)_{-2}$, if
$$a_\la v=\sum_{H}(-1)^{p_H}\frac{\la_{\bar H}}{f(H)}v_H,$$
then $v_K=-\sum_i\alpha_i\partial_i v$.


\begin{definition} The conformal dual $M^*$ of a conformal $R$-module $M$ is defined as
$$M^*=\{f_{\la}: M\rightarrow \displaywedge[\la]~|~ f_{\la}(\de_i m)=(-1)^{p_ip(f)}\lambda_i f_{\la}(m),
~{\mbox{for all}}~ m\in M\mbox{ and }i=1,\ldots,r+s\},$$
with the structure of $\inlinewedge[\deb]$-module given by $(\de_i f)_{\la}(m)=-\lambda_i f_{\la}(m)$, and
with the following $\la$-action of $R$:
$$(a_{\la}f)_{\mub}m=-(-1)^{p(a)p(f)}f_{\mub-\la}(a_{\la}m), ~a\in R, m\in M.$$
Here by $p(f)$ we denote the parity of the map $f_\la$.
\end{definition}

\begin{proposition} If $M$ is a conformal $R$-module, then $M^*$ is a conformal $R$-module.
If, in addition, $M$ is coherent, then $M^*$ is also coherent.
\end{proposition}
\begin{proof}
We need to check that properties (M1) and (M2) hold for $M^*$.
We have:
\begin{align*}
((\de_ia)_{\la}f)_{\mub}m&=-(-1)^{(p_i+p(a))p(f)}f_{\mub-\la}((\de_ia)_{\la}m)=
(-1)^{(p_i+p(a))p(f)}(-1)^{p_ip(f)}\lambda_if_{\mub-\la}(a_{\la}m)\\
&=(-1)^{p(a)p(f)}\lambda_if_{\mub-\la}(a_{\la}m).
\end{align*}

Besides,
\begin{align*} 
([\de_i,a_{\la}]f)_{\mub}m&=(\de_i(a_{\la}f))_{\mub}m-(-1)^{p_ip(a)}(a_{\la}\de_if)_{\mub}m\\
&=-\mu_i(a_{\la}f)_{\mub}m+(-1)^{p_ip(a)+p(a)(p(f)+p_i)}\de_if_{\mub-\la}(a_{\la}m)\\
&=\mu_i(-1)^{p(a)p(f)}f_{\mub-\la}(a_{\la}m)+(-1)^{p(a)p(f)}(\lambda_i-\mu_i)f_{\mub-\la}(a_{\la}m)\\
&=(-1)^{p(a)p(f)}\lambda_if_{\mub-\la}(a_{\la}m).
\end{align*}

(M1) thus follows.
As for (M2), we have:
\begin{align*}
([a_{\la},b_\mub]f)_\nub m&=-(-1)^{p(a)(p(b)+p(f))}(b_{\mub}f)_{\nub-\la}(a_{\la}m)+
(-1)^{p(b)p(f)}(a_{\la}f)_{\nub-\mub}(b_\mub m)\\
&=(-1)^{p(a)(p(b)+p(f))+p(b)p(f)}f_{\nub-\la-\mub}(b_{\mub}(a_{\la}m))-
(-1)^{p(b)p(f)+p(a)p(f)}f_{\nub-\la-\mub}(a_{\la}(b_{\mub}m))\\
&=-(-1)^{(p(a)+p(b))p(f)}f_{\nub-\la-\mub}([a_\la,b_\mub]m)\\
&=-(-1)^{(p(a)+p(b))p(f)}f_{\nub-\la-\mub}((a_\la b)_{\la+\mub}m)=((a_\la b)_{\la+\mub}f)_{\nub}m.
\end{align*}
Now assume that $M$ is coherent. Let $\yb_K a=\sum_i\alpha_i\de_{y_i}\in \mathfrak g(R)_{-2}$. Let $M_K$ be the subspace of $\inlinewedge(\la)\otimes M$ spanned by all elements of the form $\la_{\bar H}v$ for all $H\neq K$, $\ell(H)\geq \ell(K)$ and all $v\in M$. We have
\[
a_\la v=-(-1)^{p_K}\frac{\la_{\bar K}}{f(K)}\sum \alpha_i \de_iv \mod M_K.\]
Therefore
\begin{align*}
(a_\la f)_\mub v&=-(-1)^{p(a)p(f)}f_{\mub-\la}(a_\la v)\\
&= -(-1)^{p(a)p(f)}f_{\mub-\la}\big(-(-1)^{p_K}\frac{\la_{\bar K}}{f(K)}\sum_i {\alpha_i \de_iv}\big) &\mod M_K\\
&=(-1)^{p_K}\frac{\la_{\bar K}}{f(K)}\sum_i \alpha_i(\mu_i-\lambda_i)f_{\mub-\la}v &\mod M_K\\
&= (-1)^{p_K}\frac{\la_{\bar K}}{f(K)}\sum_i \alpha_i \mu_i \big( f_{\mub-\la}v\big)\big|_{\la=0} &\mod M_K\\
&= (-1)^{p_K}\frac{\la_{\bar K}}{f(K)}\sum_i \alpha_i \mu_i  (f_{\mub}v) &\mod M_K
\end{align*}
On the other hand 
\begin{align*}
(\sum_i \alpha_i \de_i f)_\mub v&=-\sum_{i}\alpha_i \mu_i (f_{\mub}v)
\end{align*}
and the proof is complete by the observation following Definition \ref{defcoherent}.
\end{proof}

\begin{proposition} Let $T: M\rightarrow N$ be a morphism of conformal $R$-modules i.e. a linear map such that:
\begin{enumerate}
\item $T(\de_i m)=(-1)^{p_ip(T)}\de_iT(m)$,
\item $T(a_{\la}m)=(-1)^{p(a)p(T)}a_{\la}T(m)$,
\end{enumerate}
then the map $T^*: N^* \rightarrow M^*$ given by: $(T^*(f))_{\la}m=-(-1)^{p(T)p(f)}f_{\la}T(m)$
is a morphism of conformal $R$-modules.
\end{proposition}
\begin{proof}
Let us first check that if $f$ lies in $M^*$ then $T^*(f)\in N^*$. Indeed, for $m\in M$  we have
\begin{align*}
 (T^*(f))_\la (\de_im)&=-(-1)^{p(T)p(f)}f_\la  (T(\de_i m))\\&=-(-1)^{p(T)(p(f)+p_i)}f_\la (\de_i T(m)) =-(-1)^{p(T)(p(f)+p_i)+p(f)p_i}\lambda_i f_\la T(m)
\end{align*}
and 
\[
 \lambda_i(T^*(f))_\la m=-(-1)^{p(f)p(T)}\lambda_i f_\la T(m).
\]

 Let us verify property (1) for $T^*$.  We have
 \begin{align*}
  (T^*(\de_i f))_\la m&=-(-1)^{p(T)(p(f)+p_i)}\de_i f_\la T(m))\\
  & =(-1)^{p(T)(p(f)+p_i)}\lambda_i f_\la T(m)
 \end{align*}
 and \[
      \de_i(T^*(f))_\la m=-\lambda_i(T^*f)_\la m=(-1)^{p(f)p(T)}\lambda_i f_\la T(m).
     \]
     Besides we have
     \[
      (T^*(a_\la f))_\mub m=-(-1)^{p(T)(p(a)+p(f))}(a_\la f)_\mub T(m)=(-1)^{p(T)(p(a)+p(f))+p(a)p(f)}f_{\mub-\la}(a_\la T(m))
     \]
     and
     \[
      (a_\la T^*(f))_\mub m = -(-1)^{p(a)(p(f)+p(T))} (T^*(f))_{\mub-\la}(a_\la m)=(-1)^{p(a)(p(f)+p(T))+p(f)p(T)} f_{\mub-\la} (T(a_\la m)).
     \]
     Property (2) for $T^*$ follows from property (2) for $T$.
     \end{proof}

\begin{proposition}\label{azioneduale} Let $M$ be a conformal $R$-module which is free and finitely-generated as a $\inlinewedge[\deb]$-module. If $\{m_i\}$ is a basis of $M$ such that
$a_{\la}m_j=\sum_k P_{jk}(\la, \deb)m_k$, then $M^*$ is a free, finitely-generated $\inlinewedge[\deb]$-module,
with basis $\{m_i^*\}$ given by: $(m_i^*)_{\la}m_k=\delta_{ik}$ and
$$a_{\la}m_i^*=-\sum_j(-1)^{p(m_i)(p(m_j)+p(m_i))}P_{ji}(\la,-\deb-\la)m_j^*.$$
\end{proposition}
\begin{proof}
The fact that $M^*$ is a free $\inlinewedge[\deb]$-module with basis $\{m_i^*\}$ is an easy verification.
By definition we have:
\begin{align*}
(a_{\la}m_i^*)_{\mub}(m_j)&=-(-1)^{p(a)p(m_i)}(m_i^*)_{\mub-\la}(a_{\la}m_j)=-(-1)^{p(a)p(m_i)}\sum_k(m_i)^*_{\mub-\la}
P_{jk}(\la,\deb)m_k\\
&=-(-1)^{p(a)p(m_i)}(m_i)^*_{\mub-\la}(P_{ji}(\la,\deb) m_i)\\
&=-(-1)^{p(m_i)(p(m_j)+p(m_i))}P_{ji}(\la,\mub-\la).
\end{align*}
The statement follows.
\end{proof}

\begin{remark} We point out that, by Proposition \ref{azioneduale}, the map
$$\varphi: (M^*)^*\rightarrow M,
\,\,(m_i^*)^*\mapsto (-1)^{p(m_i)}m_i$$
is an isomorphism of conformal $R$-modules, provided that $M$ is free and finitely generated as 
$\inlinewedge[\deb]$-module.
\end{remark}
\begin{remark}Let $M,N$ be free and finitely generated conformal $R$-modules and $T:M\rightarrow N$ be a conformal morphism. In \cite[Proposition 2.4]{BKLR} it is shown that if $R$ is of type $(1,0)$, $T$ is injective and $N/ImT$ is free as $\C[\de]$-module, then $T^*:N^*\rightarrow M^*$ is surjective and that the injectivity of $T$ is not sufficient.  The same argument applies also for $R$ of type $(r,s)$. 
On the other hand, it is easy to check that if $T$ is a surjective morphism of conformal modules then $T^*$ is always injective.
\end{remark}
Recall that we always assume that Assumptions \ref{assumption} on $R$ are satisfied, in particular,  ${\mathfrak g}(R)=\oplus_{j\geq -3}{\mathfrak g}(R)_j$ is a $\Z$-graded
Lie superalgebra of depth at most 3 with ${\mathfrak g}(R)_{-1}$ and
${\mathfrak g}(R)_{-3}$ purely odd, 
and ${\mathfrak g}(R)_{-2}$ identified with $D=\langle \de_{y_1},\ldots,\de_{y_{r+s}}\rangle$.

Let $F$ be a finite-dimensional ${\mathfrak g}(R)_{0}$-module  which we extend to
${\mathfrak g}(R)_{\geq 0}=\bigoplus_{j\geq 0}{\mathfrak g}(R)_{j}$ by letting
${\mathfrak g}(R)_{j}$, $j>0$, act trivially. We let
$$M(F)=\textrm{Ind}_{{\mathfrak g}(R)_{\geq 0}}^{{\mathfrak g}(R)}F$$
be the generalized Verma module, attached to $F$. Since, by our assumptions, 
$\inlinewedge[\deb_\yb]={\mathcal U}({\mathfrak g(R)}_{-2})$, 
and ${\mathfrak g}(R)_{-1}$ and ${\mathfrak g}(R)_{-3}$ are purely odd, we see that $M(F)$ is a free
$\inlinewedge[\deb_\yb]$-module of rank $2^n$, where $n=\dim(\mathfrak{g}(R)_{-1}+\mathfrak{g}(R)_{-3})$.

Let $\{d_1, \dots, d_n\}$ be a basis of $\mathfrak{g}(R)_{-1}\oplus \mathfrak{g}(R)_{-3}$, 
and  $\{v_1,\ldots,v_{\ell}\}$ be a basis of $F$.
For every $I\subset\{1,\dots,n\}$,
$I=\{i_1,\dots, i_k\}$ with $i_1<\dots<i_k$ we let $d_I=d_{i_1}\dots d_{i_k}\in \mathcal U(\mathfrak g(R))$.
We set $\Omega=\{1, \dots,n\}$. We observe that $\{d_Iv_h\,|\, I\subseteq \Omega,\,h\in \{1,\ldots,\ell\}\}$ is a basis of $M(F)$ as a free $\inlinewedge[\deb_\yb]$-module and we will denote by$\{(d_Iv_h)^*\,|\, I\subseteq \Omega,\,h\in \{1,\ldots,\ell\}\}$ the corresponding dual basis of $M(F)^*$ as a free $\inlinewedge[\deb_\yb]$-module.


\begin{lemma}\label{commutatore} Let $x\in \mathfrak{g}(R)_{\geq 0}$. Then for all $I\subseteq \Omega$ and all $k=1,\ldots,\ell$ we have 

\[
 xd_Iv_k\in \bigoplus_{J,h:\, |J|\leq |I|}\displaywedge [\deb_{\yb}] d_J v_h.
\]
\end{lemma}
\begin{proof}
We proceed by induction on $|I|$. If $|I|=0$ the result is clear since $xv_k\in F$. If $|I|>0$ let $I=\{i_1,\ldots,i_t\}$. Then
\[
xd_{i_1}\cdots d_{i_t} v_k=[x,d_{i_1}]d_{i_2}\cdots d_{i_t}v_k+(-1)^{p(x)}d_{i_1}xd_{i_2}\cdots d_{i_t}v_k. 
\]
Consider the first summand. If $\deg[x,d_{i_1}]<0$ is odd then $[x,d_{i_1}]$ is a linear combination of the $d_i$'s and the result follows. If $\deg[x,d_{i_1}]<0$ is even then $[x,d_{i_1}]\in \mathcal \inlinewedge[\deb_{\yb}]$ and the result also follows. If $\deg[x,d_{i_1}]\geq 0$ we can apply our induction hypothesis. 

For the second summand we can also apply the induction hypothesis so that 
\[xd_{i_2}\cdots d_{i_t}v_k=\sum_{J',h:\, |J'|\leq t-1}P_{J',h}(\deb_{\yb}) d_{J'}v_h
\]
and the result follows observing that $d_i\de_{y_j}=\pm \de_{y_j}d_i +\sum a_h d_h$.\end{proof}
%
\begin{definition}\label{shifteddual} If $\mathfrak{g}$ is a  Lie superalgebra, $\varphi: \mathfrak g\rightarrow \mathfrak{gl}(V)$ is a representation of $\mathfrak{g}$ and  $x\mapsto \chi_x\in \C$ is a character of $\mathfrak{g}$, we let $\varphi^\chi:\mathfrak{g}\rightarrow \mathfrak{gl}(V)$ be given by
	\[\varphi^\chi(x)(v)=\varphi(x)(v)+\chi_x v.
	\]
	It is clear that $\varphi^\chi$ is still a representation and we call it the $\chi$-shift of $\varphi$.
In particular, if $V$ is any $\mathfrak g$-module we call the $\chi$-shifted dual of $V$ the $\chi$-shift of the dual representation $V^*$. More explicitly, if $\{v_1,\ldots,v_n\}$ is a basis of $V$ and $\{v_1^*,\ldots,v_n^*\}$ is the corresponding dual basis of $V^*$  the $\chi$-shifted action on $V^*$ is given by 	
 
\begin{equation}\label{shiftdual}x.v_h^*=\chi_x v_h^*-(-1)^{p(x)p(v_h)}\sum_{k}v_h^*(x.v_k)v_k^*,
\end{equation}
for all $x\in \mathfrak{g}$ and $h=1,\ldots,n$.

\end{definition}
Now let $U=U(\mathfrak g(R)_{<0})$ and observe that $U$ is a graded $\mathfrak g(R)_0$-module by adjoint action. Let $d=\deg d_\Omega$ and consider the homogeneous component of degree $d$ of $U$
\[
U_d=\bigoplus_{I\subseteq \Omega} \displaywedge(\deb_\yb)_{d-\deg d_I} d_I 
\]
and note that
\[
U^\circ_d=\bigoplus_{I\neq \Omega} \displaywedge(\deb_\yb)_{d-\deg d_I} d_I
\]
is a $\mathfrak g(R)_0$-submodule of codimension 1 by Lemma \ref{commutatore}. So we have that $U_d/U_d^\circ$ is a 1-dimensional representation of $\mathfrak g(R)_0$ and we denote by $x\mapsto \rho_x\in \C$ the corresponding character. In other words the character $\rho$ is uniquely determined by the condition
\begin{equation}\label{shift1}
[x,d_\Omega]=\rho_x d_\Omega \mod U^\circ_d.
\end{equation}

\begin{lemma} \label{rho}
	Let $x\in \mathfrak g(R)_0$. Then \[\rho_x=-\str (\ad(x)_{|\mathfrak g(R)_{-1}\oplus \mathfrak g(R)_{-3}}),\]
where $\str$ denotes the supertrace of an endomorphism of a vector superspace.
\end{lemma}
\begin{proof}
	
Recall that $d_{\Omega}=d_1\cdots d_n$ where $\{d_1,\ldots,d_n\}$ is a basis of $\mathfrak g(R)_{-1}\oplus \mathfrak g(R)_{-3}$. If $n=0$ the result is trivial since $d_{\Omega}=1$. The result is also trivial if $p(x)=1$ since $\rho_x=0$ for parity reasons and $\ad(x)=0$ because  $\mathfrak g(R)_{-1}\oplus \mathfrak g(R)_{-3}$ is purely odd.

So we can assume that $p(x)=0$ and $n\geq 1$. Next observe that if $i_1,\ldots,i_t \in \{1,\ldots,n\}$ are such that $i_h=i_k$ for some $h\neq k$ then
$d_{i_1}\cdots d_{i_t}\in U^\circ _d$. So, if $[x,d_i]=\sum_j a_{ij}d_j$ we have
\begin{align*}
[x,d_{\Omega}]&=[x,d_1]d_2\cdots d_n+d_1 [x,d_2]d_3\cdots d_n+\cdots+d_1\cdots d_{n-1}[x,d_n]\\
&\equiv a_{11}d_1d_2\cdots d_n+ d_1(a_{22}d_2)d_3\cdots d_n+\cdots + d_1\cdots d_{n-1}[a_{nn}d_n] &\mod U^\circ_d\\
&\equiv (a_{11}+\cdots +a_{nn})d_{\Omega} &\mod U^\circ_d.
\end{align*}
Therefore $\rho_x=a_{11}+\cdots +a_{nn}=-\str (\ad(x)|_{\mathfrak g(R)_{-1}\oplus \mathfrak g(R)_{-3}})$.
\end{proof}	
We let $\{\de_{y_1}^*,\ldots,\de_{y_{r+s}}^*\}$ be the basis of $\mathfrak g(R)_{-2}^*$ dual to $\{\de_{y_1},\ldots,\de_{y_{r+s}}\} $. 
We let \[\pi:U=\bigoplus_I \displaywedge[\deb_\yb] d_I\rightarrow \displaywedge[\deb_\yb]d_{\Omega}\] be the natural projection. Analogously we have $M(F)=\oplus_I \inlinewedge [\deb]d_I F$ and we also denote by $\pi:M(F)\rightarrow \inlinewedge [\deb]d_\Omega F$ the corresponding projection. 
The following result is crucial in this paper.
\begin{theorem}\label{dualmodule78}  The subspace
$F_\Omega=span_{\C}\{(d_\Omega v_k)^* ~|~ k=1,\dots, \ell\}$ of $M(F)^*$ is a ${\mathfrak{g}(R)}_0$-module isomorphic to
the $\chi$-shifted dual of $F$, where the shift character is given by
\begin{equation}\label{shiftstr}\chi_x=\str(\ad(x)|_{\mathfrak g(R)_{<0}}),
\,\, x\in\mathfrak{g}(R)_0.\end{equation}
\end{theorem}
\begin{proof}  
Recall that for all $a\in R$ we have
\begin{equation}
 {a}_\la (d_Iv_k) = \sum_{M\in S_{r,s}} (-1)^{p_M}\frac{\la_{\bar M}}{f(M)}(\yb_Ma) (d_Iv_k)=\sum_{J,h}P_{I,k;J,h}(\la,\deb)d_Jv_h
\label{lambda-actionpasso0}
\end{equation}
for some polynomials $P_{I,k;J,h}(\la,\deb)$. Let $\yb_Sa\in \mathfrak g(R)_0$ and observe that $\deg a\leq 0$ is even. We make use of \eqref{lambda-actionpasso0} to compute the action of $\yb_ Sa$ on  $F_\Omega$.  We point out that $P_{I,k;\Omega,h}(\la,\deb)\neq 0$ only if $I=\Omega$. Indeed, if $\yb _Ma\in \mathfrak g(R)_{<0}$  then $\yb_Ma\in \mathfrak{g}(R)_{-2}$  and so we have
$\yb_Ma=\sum \de_{y_i}^*(\yb_Ma)\partial_{y_i}$ hence
\[
(\yb_Ma) (d_Iv_k)=\sum \de_{y_i}^*(\yb_Ma)\partial_{y_i}d_Iv_k=-\sum \de_{y_i}^*(\yb_Ma)\partial_{i}d_Iv_k;
\] 
if $\yb_Ma\in \mathfrak{g}(R)_{\geq 0}$ this follows by applying Lemma \ref{commutatore}.



In order to compute  the polynomials $P_{\Omega,k;\Omega,h}(\la, \deb)$ we notice that by \eqref{lambda-actionpasso0} and Lemma \ref{commutatore} we have
\begin{align*}
 \sum_h P_{\Omega,k;\Omega,h}(\la,\deb )d_\Omega v_h&=\sum_{M\in S_{r,s}} (-1)^{p_M}\frac{\la_{\bar M}}{f(M)}\pi((\yb_Ma)d_\Omega v_k)
 \\&=\sum_{M:\, \deg(\yb _Ma)=0,-2}(-1)^{p_M}\frac{\la_{\bar M}}{f(M)}\pi((\yb_Ma)d_\Omega v_k).
 \end{align*}
 It follows that
 \begin{align*}
 \sum_h P_{\Omega,k;\Omega,h}(\la,\deb)d_\Omega v_h=&\sum_{M:\deg \yb_Ma=-2}(-1)^{p_M} \frac{\la_{\bar M}}{f(M)} 
 \sum_i \de_{y_i}^*(\yb_Ma) \de_{y_i}d_\Omega v_k\\&+\sum _{N:\deg \yb_Na=0}(-1)^{p_N}\frac{\la_{\bar{N}}}{f(N)}\Big(\rho_{\yb_Na}v_k+(-1)^{p(d_\Omega)p(\yb_N a)}d_\Omega (\yb_Na).v_k\Big).
 \end{align*}

Substituting $(\yb_Na). v_k=\sum_h v_h^*((\yb_Na). v_k)v_h$ in the previous formula we deduce that

\begin{align*}P_{\Omega,k;\Omega,h}(\la,\deb)=&\delta_{k,h}\Big(-\sum_{M:\deg \yb_Ma=-2}\sum_{i} (-1)^{p_M}\frac{\la_{\bar M}}{f(M)}\de_{y_i}^*(\yb_Ma)\de_i+\sum_{N: \deg \yb _Na=0}(-1)^{p_N}\frac{\la_{\bar N}}{f(N)}\rho_{\yb_N a}\Big)\\&+
(-1)^{p(d_\Omega)p(\yb _Na)}\sum_{N: \deg \yb _Na=0 }(-1)^{p_N}\frac{\la_{\bar N}}{f(N)}
v_h^*(\yb_Nav_k).
\end{align*}

By Proposition \ref{azioneduale} it follows that
\begin{align}\label{lambda-action}
{a}&_\la(d_\Omega v_h)^*  =\\
\nonumber &\sum_{M:\, \deg \yb_Ma=-2}(-1)^{p_M}\frac{\la_{\bar M}}{f(M)}\sum_{i}\de_{y_i}^*(\yb_Ma)(-\de_i-\lambda_i)(d_{\Omega} v_h)^* - \sum_{N:\, \deg \yb _N a=0}(-1)^{p_N}\frac{\la_{\bar N}}{f(N)}\rho_{\yb_N a}(d_{\Omega} v_h)^*\\
\nonumber & -\sum_k\sum_{N: \deg \yb _Na=0}(-1)^{p_N}\frac{\la_{\bar N}}{f(N)}(-1)^{p(d_\Omega v_h)(p(d_\Omega v_h)+p(d_\Omega v_k))}(-1)^{p(d_\Omega)p(\yb _N a)}v_h^*(\yb_N a\,v_k)(d_{\Omega}v_k)^*.
\end{align}
Observe that in each summand of the last sum we have a factor $v_h^*(\yb_N a\, v_k)$ so that we can assume $p(v_h)+p(v_k)=p(\yb _N a)$ in order to simplify the sign
\[
(-1)^{p(d_\Omega v_h)(p(d_\Omega v_h)+p(d_\Omega v_k))}(-1)^{p(d_\Omega)p(\yb _N a)}=(-1)^{p(v_h)p(\yb_Na)}.
\]
Now recall that the action of $\yb_Sa\in\mathfrak{g}(R)_0$ on $(d_\Omega v_h)^*$ can be obtained from  (\ref{lambda-action}) thanks to Proposition \ref{modcorr}. This immediately implies that
 $F_\Omega$ is a $\mathfrak{g}(R)_0$-submodule
of $M(F)^*$. More precisely, for $\yb_Sa\in \mathfrak{g}(R)_0$ we have:
\begin{align}\label{gzeroaction}
\yb_Sa.(d_\Omega v_h)^*=&\Big(-\sum_{i,M:\, iM\sim S}(-1)^{p_i}\varepsilon_{(i,M;S)}\,m_i(S)\de_{y_i}^*(\yb_Ma)-\rho_{\yb_Sa}\Big) (d_\Omega v_h)^*\\
\nonumber&-(-1)^{p(v_h)p(\yb _S a)}\sum_kv_h^*(\yb_Sa\,v_k)(d_{\Omega}v_k)^*.
\end{align}
Recalling that $iM\sim S$ implies $y_i \yb_M=\varepsilon_{i,M;S}\yb _S$ we have $\de_{y_i}(\yb _S)= \varepsilon_{i,M;S} m_i(S)\yb _M$ and so we have
\begin{align*}
\yb_Sa.(d_\Omega v_h)^* &=\big( -\sum_{i} (-1)^{p_i}\de_{y_i}^*(\de_{y_i}\yb _S a)-\rho_{\yb _S a}\big)(d_\Omega v_h)^*-(-1)^{p(v_h)p(\yb _S a)}\sum_kv_h^*(\yb_Sa\,v_k)(d_{\Omega}v_k)^*\\
&=\big( \sum_{i} (-1)^{p_i}\de_{y_i}^*[\yb_Sa,\de_{y_i}] -\rho_{\yb _S a}\big)(d_\Omega v_h)^*-(-1)^{p(v_h)p(\yb _S a)}\sum_kv_h^*(\yb_Sa\,v_k)(d_{\Omega}v_k)^*\\
&=\str\big (\ad({\yb_S a})_{|\mathfrak g(R)_{<0}}\big)(d_\Omega v_h)^*-(-1)^{p(v_h)p(\yb _S a)}\sum_kv_h^*(\yb_Sa\,v_k)(d_{\Omega}v_k)^*,
\end{align*}
by Lemma \ref{rho}.
The result follows by  \eqref{shiftdual}.
\end{proof}


\begin{proposition}\label{sing} The $\mathfrak g(R)_0$-module $F_\Omega$ is annihilated by $\mathfrak g(R)_{>0}$.
\end{proposition}
\begin{proof}
By the same argument used in the proof of Theorem \ref{dualmodule78} if the degree of $a$ is positive then $a_{\la}(d_{\Omega}v_{h})^*=0$; moreover, 
if $a$ has even non-positive degree then $\yb_Ma.(d_{\Omega}v_{h})^*=0$ by (\ref{lambda-action}).

Now let us assume $\deg(a)=-2g+1$ for some positive integer $g$. If $\ell(M)=g-1,g-2$ then we can write $\yb_Ma=\sum_{i=1}^n\alpha_{M,i}(a)d_i$ for some
$\alpha_{M,i}(a)\in\C$. 
From (\ref{lambda-actionpasso0}) we obtain
$$P_{I,k;\Omega,h}(\la,\deb)=\begin{cases}
\displaystyle \sum_{\ell(M)=g-1,g-2}(-1)^{i-1}\alpha_{M,i}(a)\frac{\la^M}{M!}\delta_{kh} & ~~{\mbox{if}}~~ I=\Omega \setminus \{i\},  {\mbox{for some}}~~ 1\leq i\leq n,\\
0 & ~~{\mbox{otherwise.}}
\end{cases}
$$

It follows that 
\begin{align*}
a_{\la}(d_{\Omega}v_{h})^*&=-\sum_{I,k}(-1)^{p(d_\Omega v_h)(p(d_\Omega v_h)+p(d_I v_k))}P_{I,k;\Omega,h}(\la, -\deb-\la)(d_Iv_k)^*\\
&=-\sum_{\ell(M)=g-1,g-2} \sum_i(-1)^{p(d_\Omega v_h)(p(d_\Omega v_h)+p(d_{\Omega\setminus\{i\}} v_k))} (-1)^{i-1} \alpha_{M,i}(a)\frac{\la_M}{M!}(d_{\Omega \setminus\{i\}}v_h)^*.
\end{align*}
It follows that $\yb_Ma.(d_{\Omega}v_{h})^*=0$ for $\ell(M)\geq g$.
\end{proof}
For the 1-dimensional  $\mathfrak g(R)_0$-module $\chi$ given by \eqref{shiftstr}, we denote the $\chi$-shifted dual of   $F$ by $F^\vee$.
The following is our main result.
\begin{theorem}\label{main} Let $R$ be a Lie conformal superalgebra of type $(r,s)$ satisfying Assumptions \ref{assumption}. Let $F$ be a finite-dimensional $\mathfrak{g}(R)_0$-module. Then $M(F)^*$ is isomorphic to $M(F^\vee)$ as a $\mathfrak g(R)$-module.
\end{theorem}

\begin{proof} We use the same notation as above and consider the basis $\{v_1^*,\dots, v_\ell^*\}$ of $F^\vee$ dual to the basis
	$\{v_1,,\dots,v_\ell\}$. By Theorem \ref{dualmodule78} and Proposition \ref{sing}, we can define a morphism $\varphi: M(F^\vee)\rightarrow M(F)^*$ of $\mathfrak g(R)$-modules by extending the $\mathfrak g (R)_{\geq 0}$-modules isomorphism $\varphi:F^\vee\rightarrow F_\Omega$ given by $\varphi(v_k^*)=(d_\Omega v_k)^*$ for all $k=1,\ldots,\ell$, in the following natural way:
	\[\varphi(u \otimes v)=(-1)^{p(u)p(d_\Omega)}u\varphi (v)
	\]
	for all $u\in U(\mathfrak g(R))$ and $v\in F^\vee$.
	
	Since $M(F^\vee)$ and $M(F)^*$ are free $\inlinewedge[\deb]$-modules of the same rank, we will prove that $\varphi$ is in fact an isomorphism by
	showing that it is  surjective. To this aim it is sufficient to show that the $\mathfrak g(R)$-submodule $S$ of $M(F)^*$ generated by
	$F_\Omega$ is the whole $M(F)^*$. 
	Let $a\in R$ be an element of negative odd degree, say $\deg a=-2g+1$. Recall that by Lemma \ref{commutatore} if $\deg (\yb_Ma)>0$ we have 
	\[\yb_Ma\, d_Iv_k\in \bigoplus_{J,h:\, |J|<|I|}\displaywedge[\deb_\yb] d_Jv_h,\] and if $\deg (\yb_Ma)<0$ then $\yb_Ma=\sum_i \alpha_{i,M}(a)d_i$. Therefore
	\[a_\la d_Iv_k=\sum_{\ell(M)=g-1,g-2}(-1)^{p_M}\frac{\la_{\bar M}}{M!}\sum_i\alpha_{i,M}(a)d_id_I v_k+\sum_h \sum_{J:\, |J|<|I|}P_{I,k;J,h}(\la, \deb)d_Jv_h.
	\] 
	In particular, if $J=\{j_1,\ldots,j_r\}$ with $j_1<j_2<\cdots<j_r$ and $|J|\geq |I|$, then
	\[
	 P_{I,k;J,h}(\la,\deb)=\begin{cases}
	                        \delta_{hk}\displaystyle\sum_{\ell(M)=g-1,g-2}(-1)^{p_M}\frac{\la_{\bar M}}{M!}(-1)^{i-1}\alpha_{i,M}(a)&\textrm{if }I=J\setminus \{j_i\}\, \textrm{ for some $i$}\\
	                        0&\textrm{otherwise.}
	                       \end{cases}
	\]

	It follows that
	\begin{align}
	\nonumber  a_\la (d_J v_h)^*=&\sum_{i}-(-1)^{p(d_J)(p(d_J)+p(d_{J\setminus \{j_i\}}))}\sum_{\ell(M)=g-1,g-2}(-1)^{p_M}\frac{\la_{\bar M}}{M!} (-1)^{i-1}\alpha_{i,M}(a) (d_{J\setminus\{j_i\}}v_h)^*\\&-\sum_{k}\sum_{I:\, |I|>|J|}(-1)^{p(d_J)(p(d_J)+p(d_I))}P_{I,k;J,h}(\la, -\la-\deb)(d_Iv_k)^*.\label{hkhk}
	\end{align}
We show by reverse induction on $|I|$ that $S$ contains all elements $(d_Iv_k)^*$. The first step of the induction consists of the elements in $F_\Omega$ which belong to $S$ by definition.

So let $I\subsetneq \Omega$ and let $J$ be such that $I\subsetneq J$ and $|J|=|I|+1$. If $J=\{j_1,\ldots,j_r\}$ there exists $i$ such that $I=J\setminus \{j_i\}$.  By induction hypothesis all elements $(d_Jv_h)^*\in S$. Now we observe that there exist $a\in R$ and $N$ such that $d_{j_i}=\yb_Na$. Indeed, by definition of $\mathfrak g(R)$,  $d_{j_i}$ can be expressed as follows:
\[
 d_{j_i}=\sum_r \yb_{M_r} a_r.
\]
Let $N$ be such that $\yb_N$ is a non-zero common multiple of all $\yb_{M_r}$. Then we can write
\[
 d_{j_i}=\sum_r \yb_{M_r} a_r=\sum_r \beta_r \yb_N (\deb_{\yb})_{I_r} a_r={\yb}_N \sum_r \beta_r (-1)^{\ell(I_r)}\deb_{I_r}a_r={\yb}_Na,
\]
where $\beta_r$ are suitable constants and $I_r$ are suitable sequences of  indices.
By Equation \eqref{hkhk}, observing that $\alpha_{j_i,N}(a)=1$, we have
\[
 \yb_Na\, (d_Jv_h)^*=-(-1)^{p(d_J)(p(d_J)+p(d_I))}(-1)^{i-1}(d_Iv_h)^*+Z,
\]
where \[Z\in \sum_k\sum_{L:\, |L|>|J|}\displaywedge[\deb](d_Lv_k)^*.\]
This completes the proof.
\end{proof}

\section{The Lie conformal superalgebra of type $W$.}\label{section4}
We denote by $W(r,s)$ the Lie superalgebra of derivations of $\inlinewedge[[\xb]]$, where, as usual, ${\xb}=(x_1,\ldots,x_{r+s})$ and $x_1,\ldots,x_r$ are even  and $x_{r+1},\ldots,x_{r+s}$ are odd variables. In this section we realize $W(r,s)$ as the annihilation superalgebra associated to a Lie conformal superalgebra of type $(r,s)$ which satisfies Assumptions \ref{assumption}.
\begin{definition}\label{labracketW}
	We denote by $RW(r,s)$ the free $\inlinewedge[\deb]$-module with even generators $a_1,\ldots,a_{r}$ and odd generators $a_{r+1},\ldots,a_{r+s}$, all of degree $-2$, and $\la$-bracket given by\[
	[{a_i}_\la a_j]=(\de_i+\lambda_i)a_j+a_i\lambda_j\]
	and extended on the whole $RW(r,s)$ by properties \eqref{defsupconf1} and \eqref{defsupconf2} of Definition \ref{defsupconf}.
\end{definition}
\begin{proposition}\label{4.2}
	$RW(r,s)$ is a $\Z$-graded Lie conformal superalgebra of type $(r,s)$ satisfying Assumptions \ref{assumption}.
\end{proposition}
\begin{proof} It is sufficient to verify that conformal skew-symmetry and  Jacobi identity hold for the generators $a_i$.
We first verify the conformal skew-symmetry. We have	
	\begin{align*}
	[{a_j}_\la a_i]&= (\de_j+\lambda_j)a_i+a_j\lambda_i\\
	&= (-1)^{p_ip_j}(a_i(\de_j+\lambda_j)+\lambda_ia_j)\\
	&= (-1)^{p_ip_j}(-(\de_i-\lambda_i-\de_i)a_j-a_i(-\de_j-\lambda_j))\\
	&=-(-1)^{p_ip_j}[{a_i}_{-\la-\deb}a_j].
	\end{align*}
The Jacobi identity can be verified similarly. Namely we have:
	\begin{align*}
	[{a_i}_\la [{a_j}_\mub a_k]]&=[{a_i}_\la (\de_j+\mu_j)a_k+a_j\mu_k]\\
	&=(-1)^{p_ip_j}\mu_j [{a_i}_\la a_k]+[{a_i}_\la \de_ja_k]+[{a_i}_\la a_j]\mu_k\\
	&= (-1)^{p_ip_j}\mu_j((\de_i+\lambda_i)a_k+a_i\lambda_k)+(-1)^{p_ip_j}(\de_j+\lambda_j)[{a_i}_\la a_k]\\
	&+((\de_i+\lambda_i)a_j+a_i\lambda_j)\mu_k\\
	&= (\de_i+\lambda_i)\mu_ja_k+a_i\mu_j\lambda_k
	 +(\de_i+\lambda_i)(\de_j+\lambda_j)a_k
	+a_i(\de_j+\lambda_j)\lambda_k+(\de_i+\lambda_i)a_j\mu_k\\
	& +a_i\lambda_j\mu_k;
	\end{align*}
	
	\begin{align*}
	[[{a_i}_\la a_j]_{\la+\mub}a_k]&=[{(\de_i+\lambda_i)a_j+a_i\lambda_j\,}_{\la+\mub} a_k]\\
	&=-\mu_i[{a_j}_{\la+\mub}a_k]+(-1)^{p_ip_j}\lambda_j[{a_i}_{\la+\mub}a_k]\\
	&=-\mu_i((\de_j+\lambda_j+\mu_j)a_k+a_j(\lambda_k+\mu_k))+(\de_i+\lambda_i+\mu_i)\lambda_ja_k+a_i\lambda_j(\lambda_k+\mu_k);
	\end{align*}
	
	\begin{align*}
	(-1)^{p_ip_j}[{a_j}_\mub[{a_i}_\la a_k]]&=(-1)^{p_ip_j}[{a_j}_\mub (\de_i+\lambda_i)a_k+a_i\lambda_k]\\
	&=(\de_i+\mu_i+\lambda_i)((\de_j+\mu_j)a_k+a_j\mu_k)+a_i(\de_j+\mu_j)\lambda_k+\mu_ia_j\lambda_k.
	\end{align*}
	The fact that $RW(r,s)$ satisfies Assumptions \ref{assumption} is straightforward.
	
	\end{proof}
	\begin{proposition}\label{isom}The map
	\[
	 \varphi: \mathfrak g(RW(r,s))\rightarrow W(r,s)
	\]
	given by $\yb_M a_i \mapsto - \xb_M \de_{x_i}$
	is a $\Z$-graded Lie superalgebra isomorphism, where the $\Z$-gradation on $W(r,s)$ is given by $\deg x_i=-\deg \de_{x_i}=2$.

	\end{proposition}
	\begin{proof}
	In the Lie superalgebra $W(r,s)$ one has:
	\begin{align*} 
	[\xb_{M} \de_{x_i}, \xb_{N} \de_{x_j}]=\xb_{M}(\de_{x_i}\xb_{N})\de_{x_j}-(-1)^{(p_i+p_M)(p_j+p_N)}\xb_{N}(\de_{x_j}\xb_{M})\de_{x_i}.
	\end{align*}
	On the other hand, by Definition \ref{annihilationalgebra}, we have:
	\begin{align*}
	[\yb_M a_i, \yb_N a_j]&=(-1)^{p_jp_N}\sum_{K}\frac{1}{f(K)}(\yb_M (\de_{\yb}~)_K)({a_i}_Ka_j)\yb_N\\
	&=(-1)^{p_jp_N}(\yb_M(\de_ia_j)\yb_N+(\yb_M\de_{y_i})a_j\yb_N+(\yb_M\de_{y_j})((-1)^{p_ip_j}a_i)\yb_N\\
	&=(-1)^{p_jp_N}(-(-1)^{p_ip_M+p_jp_N}(\de_{y_i}\yb_M\yb_Na_j)+(-1)^{p_ip_M+p_jp_N}(\de_{y_i}\yb_M)\yb_Na_j\\
&	+(-1)^{p_ip_j+p_jp_M+p_ip_N}(\de_{y_j}\yb_M)\yb_Na_i)\\
	&=-\yb_M(\de_{y_i}\yb_N)a_j+(-1)^{(p_i+p_M)(p_j+p_N)}\yb_N(\de_{y_i}\yb_M)a_j.
	\end{align*}
	 \end{proof}

\begin{remark}Let $R=RW(r,s)$.
  We observe that $\mathfrak g(R)_0 \cong \mathfrak{gl}(r,s)$,  hence the center of the even part of $\mathfrak g(R)_0$ is spanned by $c_0=\sum_{i=1}^r
y_ia_i$ and $c_1=\sum_{i=r+1}^{r+s}
y_ia_i$.
We have $ \chi_{c_0}=-r$ and  $ \chi_{c_1}=s.$

We point out that this is consistent with the classification of degenerate Verma modules of $W(1,s)$ given in \cite{BKLR} that we recall below. In this case let us denote by $F(\alpha_0,\alpha_1; \beta_1,\ldots,\beta_{s-1})$ the irreducible $\mathfrak{gl}(1,s)$-module with highest weight $(\alpha_0,\alpha_1; \beta_1,\ldots,\beta_{s-1})$ with respect to the elements $x_1\de_{x_1}$, $\sum_{i=2}^{s+1}x_i \de_{x_i}$, $x_2\de_{x_2}-x_3\de_{x_3},\ldots,x_{s}\de_{x_s}-x_{s+1}\de_{x_{s+1}}$. 

We observe that for $k>0$, the dual of the $\mathfrak{gl}(1,s)$-module $F(0,-k;0,\ldots,0,k)$ is $F(1,k-1;k-1,0,\ldots,0)$. Indeed the module $F(1,k-1;k-1,0,\ldots,0)$ consists of $k$-forms in the indeterminates $x_1,\ldots,x_{s+1}$ with constant coefficients, its highest weight vector is $dx_1(dx_2)^{k-1}$ and its lowest weight vector is $dx_{s+1}^k$. It follows that the shifted dual of $F(0,-k;0,\ldots,0,k)$ is $F(0,s+k-1;k-1,0,\ldots,0)$. Moreover, the module $F(0,0;0,\ldots,0)$ is self-dual and so its shifted dual is $F(-1,s;0,\ldots,0)$.

In this notation we have the following two sequences of morphisms of  generalized Verma modules:

\[
 \begin{array}{ccccccc}M(0,0;0,\ldots,0)& \leftarrow & M(0,-1;0,\ldots,0,1)& \leftarrow & M(0,-2;0,\ldots,0,2)& \leftarrow &   \cdots \\
 &&&&&&\\
 M(-1,s;0,\ldots,0)& \rightarrow & M(0,s;0,0,\ldots,0)& \rightarrow & M(0,s+1;1,0\ldots,0)& \rightarrow &   \cdots,
 \end{array}
\]
which contain all non-trivial morphisms. Note that these two sequences are dual to each other.

\end{remark}

\section{The finite simple Lie conformal superalgebras of type $(1,0)$}\label{section5}

It is known that all infinite-dimensional simple linearly compact Lie superalgebras of growth 1 (rather their universal central extensions) are annihilation Lie superalgebras of simple Lie conformal superalgebras of type $(1,0)$, which have finite rank over $\C[\de]$
\cite{FK}. These Lie superalgebras are
\begin{equation}
W(1,n), ~S'(1,n), ~K(1,n), ~K'(1,4) ~{\mbox{and}}~ E(1,6),
\label{(1,0)}
\end{equation}
except that here by $K'(1,4)$ we mean the universal central extension of $[K(1,4), K(1,4)]$.
The corresponding simple Lie conformal superalgebras are respectively
\begin{equation}
W_n, ~S_n, ~K_n, ~K'_4 ~{\mbox{and}}~ CK_6
\label{(1,0)Lie}
\end{equation}
(in fact there are many Lie conformal superalgebras with the annihilation Lie superalgebra $S(1,n)$ \cite{FK}).

All Lie superalgebras (\ref{(1,0)}) are $\Z$-graded by their principal gradation described in
\cite{CK2}, and these gradations are induced from $\Z$-gradations of the corresponding Lie conformal superalgebras. For example
 the principal $\Z$-gradation of $W(1,n)$ and $S(1,n)$ is defined as in Proposition 
\ref{isom}, and it is induced from the $\Z$-gradation of $RW(r,s)$ from Proposition \ref{4.2}.

The linearly compact Lie superalgebra $K(1,n)$ is identified with the linearly compact superspace
$$\displaywedge(\xib)[[y]], ~\xib=(\xi_1, \dots, \xi_n), ~ p(\xi_i)=\bar{1}, ~p(y)=\bar{0}.$$
The bracket is given by the formula
\begin{equation}\label{Kbracket}
[\phi,\psi]= \Big(2\phi - \sum_{i=1}^{n} \xi_i \de_{\xi_i} \phi \Big)(\de_y \psi)- (\de_y \phi)\Big(2\psi- \sum_{i=1}^{n} \xi_i \de_{\xi_i} \psi)+(-1)^{p(\phi)}\sum_{i=1}^n  (\de_{\xi_i}\phi)(\de_{\xi_i} \psi),
\end{equation}
and the principal $\Z$-gradation is
\begin{equation}
\deg(\xi_{i_1}\dots\xi_{i_s}t^m)=s+2m-2,
\label{Kgrading}
\end{equation}
hence it is of the form $K(1,n)=\prod_{k\geq -2}\mathfrak{g}_k$, where $\mathfrak{g}_{-2}=\C1$,
$\mathfrak{g}_{-1}={\mbox{span}}\{\xi_1,\dots,\xi_n\}$, $\mathfrak{g}_{0}={\mbox{span}}\{\xi_i\xi_j\}\oplus\C y$.
It follows from (\ref{Kbracket}) that $\mathfrak{g}_0$ is isomorphic to $\mathfrak{so}_n\oplus\C y$, where $y$ is its 
central element and $\ad y$ defines the $\Z$-gradation;
the adjoint representation of $\mathfrak{so}_n$ is trivial on $\mathfrak{g}_{-2}$ and standard on $\mathfrak{g}_{-1}$. 
Hence, by Theorem \ref{dualmodule78}, the shift character $\chi$ on $\mathfrak{g}_0$ is as follows:
\begin{equation}
\chi_{|{\mathfrak{so_n}}}=0, ~~\chi_y=n-2.
\label{Kshift}
\end{equation}

The Lie superalgebras $K'(1,4)$ and $E(1,6)$ with their principal gradation are $\Z$-graded subalgebras of $K(1,4)$ and $K(1,6)$ respectively, with the same non-positive part. Hence the shift character in these cases
is given by \eqref{Kshift} with $n=4$ and 6, respectively. Note also that on the central element $c$ of the universal central extension of $K'(1,4)$ the shift character vanishes.

This description of the shift character shows that the complexes of degenerate Verma modules described in
\cite{BKL}, \cite{B} and \cite{BKL1} for $K(1,n)$, $K'(1,4)$, and $E(1,6)$, respectively, are mapped to each other under duality described by Theorem \ref{main}.

Of course, in order to be able to apply Theorem \ref{main}, we need to check that the Lie superalgebras in question
are annihilation superalgebras of Lie conformal superalgebras, and that Assumptions \ref{assumption}
hold. But this holds for $K(1,n)$ and $K'(1,4)$ by \cite{FK} and for $E(1,6)$ by \cite{CK}.

For example, the Lie conformal superalgebra with annihilation Lie superalgebra $K(1,n)$,
as an $\C[\de]$-module, is $K_n=\C[\de]\otimes\inlinewedge(\xib)$, and the 
$\lambda$-bracket between $f=\xi_{i_1}\dots\xi_{i_k}$ and $g=\xi_{j_1}\dots\xi_{j_h}$  is given in \cite{FK} by the following formula:
\[
[f_\lambda g]= (k-2) \de(fg)+(-1)^k \sum_{i=1}^n (\de_{\xi_i}f) (\de_{\xi_i}g)+\lambda (k+h-4)fg.\]
The principal gradation on $K(1,n)$ is induced by the following $\Z$-gradation on $K_n$:
$$\deg (\xi_{i_1}\cdots \xi_{i_s})=s-2,~\deg(\de)=\deg(\lambda)=-2.$$ 

\section{The Lie conformal superalgebra $RE(5,10)$.}\label{section6}
In this section we introduce, following \cite{CK2}, the exceptional linearly  compact infinite-dimensional Lie superalgebra $E(5,10)$
and realize it as the annihilation superalgebra of a Lie conformal superalgebra of type $(5,0)$.

The even part of $E(5,10)$ consists of zero-divergence vector fields in five (even) indeterminates $x_1,\ldots,x_5$, i.e., 
\[E(5,10)_{\bar{0}}=S_5=\{X=\sum_{i=1}^5f_i\partial_{x_i} ~|~ f_i\in\C[[x_1,\dots,x_5]], \textrm{div}(X):=\sum_i\frac{\de{f_i}}{\de{x_i}}=0\},\]
and $E(5,10)_{\bar{1}}=\Omega^2_{cl}$ consists of closed two-forms in the five indeterminates $x_1,\ldots,x_5$.
The bracket between a vector field and a form is given by the Lie derivative and for $f,g\in \C[[x_1,\dots,x_5]]$ we have
$$[fdx_i\wedge dx_j,g dx_k\wedge dx_l]=\varepsilon_{ijkl}fg\partial_{x_{t_{ijkl}}}$$ 
where, for $i,j,k,l\in \{1,2,3,4,5\}$, $\varepsilon_{ijkl}$ and $t_{ijkl}$ are defined as follows: if $|\{i,j,k,l\}|=4$ we let $t_{ijkl}\in \{1,2,3,4,5\}$ be such that $|\{i,j,k,l,t_{ijkl}\}|=5$ and $\varepsilon_{ijkl}$ be the sign of the permutation 
$(i,j,k,l,t_{ijkl})$. If $|\{i,j,k,l\}|<4$ we let  $\varepsilon_{ijkl}=0$. 
From now on we shall denote $dx_i\wedge dx_j$ simply by $d_{ij}$.

The Lie superalgebra $E(5,10)$ has a consistent $\Z$-gradation of depth 2, called principal, where,
for $k\geq 0$,
\begin{align*}
E(5,10)_{2k-2}&=\langle f\partial_{x_i} ~|~i=1,\dots,5, f\in\C[[x_1,\dots, x_5]]_{k}\rangle\cap S_5\\
E(5,10)_{2k-1}&=\langle fd_{ij} ~|~ i,j=1,\dots,5, f\in\C[[x_1,\dots, x_5]]_{k}\rangle\cap\Omega^2_{cl}
\end{align*}
where 
by $\C[[x_1,\dots, x_5]]_{k}$ we denote the homogeneous component of $\C[[x_1,\dots, x_5]]$ of degree $k$.

Note that $E(5,10)_0\cong \mathfrak{sl}_5$, and its modules $E(5,10)_{-1}$ and
$E(5,10)_{-2}$ are isomorphic to the exterior square and to the dual, respectively, of the standard $\mathfrak{sl}_5$-module.

\begin{definition}\label{hatRE}
 We denote by $\widehat{RE}(5,10)$ the free $\C[\deb]$-module (where $\deb=(\de_1,\ldots,\de_5)$) generated by
 even elements $\de_{x_i}$ and odd elements $d_{jk}$ with $i,j,k\in\{1,2,3,4,5\}$, $j<k$.
 We set $d_{jj}=0$ and $d_{jk}=-d_{kj}$ if $j>k$. 
 We consider on $\widehat{RE}(5,10)$
 the following $\la$-bracket:
 \begin{itemize}
  \item $[{\de_{x_i}}_\la \de_{x_j}]= -(\de_i+\lambda_i)\de_{x_j}-\lambda_j\de_{x_i}$;
  \item $[{\de_{x_i}}_\la d_{jk}]=-(\de_i+\lambda_i)d_{jk}+\delta_{ij}\sum_{h}\lambda_hd_{hk}
  -\delta_{ki}\sum_{r}\lambda_rd_{rj}$;
  \item $[{d_{jk}}_\la {\de_{x_i}}]=-\lambda_id_{jk}+\delta_{ij}\sum_{h}(\lambda_h+\de_h)d_{hk}
  -\delta_{ki}\sum_{r}(\lambda_r+\de_r)d_{rj}$
  \item $[{d_{ij}}_\la d_{hk}]=[{d_{ij}},d_{hk}]$
 \end{itemize}
 extended on the whole $\widehat{RE}(5,10)$ by properties \eqref{defsupconf1} and \eqref{defsupconf2} of Definition \ref{defsupconf}.
 \end{definition}
 \begin{remark}
  One can verify that the $\la$-bracket defined on $\widehat{RE}(5,10)$ satisfies the conformal skew-symmetry in 
   Definition \ref{defsupconf} but it does not satisfy the conformal Jacobi identity.
   \end{remark}

Note that the $\la$-bracket $[{\de_{x_i}}_\la \de_{x_j}]$ differs in sign from that given in Definition \ref{labracketW}, but is more natural in this context.  In order to explain how Definition \ref{hatRE} arises we make a short detour on
formal distribution algebras.
We adopt the same notation and terminology as in \cite{K1}. All variables in this context are even. For this 
reason we indicate monomials with the usual multi-exponent notation, namely, for $N=(n_1, \dots, n_5)$ and
$\xb=(x_1,\ldots,x_5)$, $\xb^N=x_1^{n_1}\cdots x_5^{n_5}$.
For two variables $w,z$ define the formal $\delta$-function
\[\delta(z-w)=\sum_{i,j\,|\, i+j=-1}w^iz^j.\]
Recall that it satisfies the following properties:
\begin{itemize}
\item[(i)] $\delta(z-w)=\delta(w-z)$,
\item[(ii)] $\de_z\delta(z-w)=-\de_w\delta(z-w)$,
\item[(iii)] $f(z)\delta(z-w)=f(w)\delta(z-w)$ for any formal distribution $f(z)$.
\end{itemize}
We let  $\wb=(w_1,\ldots,w_5)$, $\zb=(z_1,\ldots,z_5)$ and $\la=(\lambda_1,\ldots,\lambda_5)$,
and introduce the formal $\delta$-function in 5 variables (the discussion below holds for any finite number of variables)
\[
\delta (\xb-\zb)= \prod_i \delta(x_i-z_i).
\]
The properties of $\delta(z-w)$ imply similar properties of it:
\begin{equation}\label{cataliz1}
f(\xb)\delta(\xb-\wb)=f(\wb)\delta(\xb-\wb), ~{\mbox{for any formal distribution}}~ f~
{\mbox{in 5 variables,}}
\end{equation}
in particular:
\begin{equation}\label{cataliz}
\delta(\xb-\zb)\delta(\xb-\wb)=\delta(\wb-\zb)\delta(\xb-\wb).
\end{equation}
Also
\begin{equation} \label{derdelta}
\de_{x_i}\delta(\xb-\wb)=-\de_{w_i}\delta(\xb-\wb),
\end{equation}
and $\delta(\xb-\wb)=\delta(\wb-\xb)$.
\begin{lemma}\label{prel}
We have
\[
\delta(\xb-\zb)\de_{x_i}\delta(\xb-\wb)=-\delta(\wb-\zb)\de_{w_i}\delta(\xb-\wb)-(\de_{w_i}\delta(\wb-\zb))\delta(\xb-\wb).
\]
\end{lemma}
\begin{proof}
Using Equations \eqref{derdelta} and \eqref{cataliz} we have
\begin{align*}
\delta(\xb-\zb)\de_{x_i}\delta(\xb-\wb)&=-\delta(\xb-\zb)\de_{w_i}\delta(\xb-\wb)\\
&=-\de_{w_i}\big( \delta(\xb-\zb)\delta(\xb-\wb)\big)\\
&=-\de_{w_i}\big(\delta(\wb-\zb)\delta(\xb-\wb)\big).\end{align*}
\end{proof}

Now we consider the free $\C[[x_{1},\ldots,x_5]][x_1^{-1},\ldots,x_5^{-1}]$-module $A_5$ with basis  $\de_{x_i}$, $i=1,\ldots,5$, and $d_{ij}$, $1\leq i<j\leq 5$. Define the skew-supersymmetric bracket on $A_5$ by the same formulas as for $E(5,10)$.
Then $A_5$ is a superalgebra containing $E(5,10)$ as a subalgebra, but the bracket 
on $A_5$ does not satisfy the Jacobi identity.

Recall that two formal distributions $a(\zb),b(\zb)\in A_5[[\zb^{\pm1}]]$, i.e.\ bilateral series with coefficients in $A_5$ in the variables $z_1,\ldots,z_5$, are called \emph{local} if 
\[
[a(\zb),b(\wb)]=\sum_{N\in \mathbb Z_+^5}\frac{1}{N!}c_N(\wb)\de_\wb^{N}\delta(\wb-\zb),
\]
where $c_N(\wb)$ are formal distributions in $A_5[[\wb^{\pm 1}]]$ and only a finite number of $c_N(\wb)$ are nonzero. 
We say in this case that $[a(\zb)_{(N)}b(\zb)]=c_N(\zb)$ is the $N$-product of $a(\zb)$ and $b(\zb)$ and we define their $\la$-bracket by
\[
[a(\zb)_\la b(\zb)]=\sum_N \frac{1}{N!}c_N(\zb) \la^{N}.
\]
For all $a\in A_5$ we let $a[\zb]=\delta(\xb-\zb)a \in A_5[[\zb^{\pm 1}]]$. It follows from the lemmas below  that these formal distributions are pairwise local.
\begin{lemma}\label{deiladej}
For all $i,j\in \{1,2,3,4,5\}$ we have
\[
[\de_{x_i}[\zb],\de_{x_j}[\wb]]=-\de_{w_i}(\de_{x_j}[\wb])\delta(\wb-\zb)-\de_{x_j}[\wb]\de_{w_i}\delta(\wb-\zb)-\de_{x_i}[\wb]\de_{w_j}\delta(\wb-\zb).
\]
In particular
\[
[\de_{x_i}[\zb]_\la \de_{x_j}[\zb]]=-\de_{z_i}(\de_{x_j}[\zb])-\de_{x_j}[\zb]\lambda_i-\de_{x_i}[\zb]\lambda_j.
\]

\end{lemma}
\begin{proof}
We have

\begin{align*}
[\de_{x_i}[\zb],\de_{x_j}[\wb]]&=[\delta(\xb-\zb)\de_{x_i},\delta(\xb-\wb)\de_{x_j}]\\
&=\delta(\xb-\zb)(\de_{x_i}\delta(\xb-\wb))\de_{x_j}-\delta(\xb-\wb)(\de_{x_j}\delta(\xb-\zb))\de_{x_i}\\
&= -\delta(\wb-\zb)(\de_{w_i}\delta(\xb-\wb))\de_{x_j}-(\de_{w_i}\delta(\wb-\zb))\delta(\xb-\wb)\de_{x_j}\\
&\hspace{5mm} -\delta(\xb-\wb)(\de_{w_j}\delta(\wb-\zb))\de_{x_i},
\end{align*}
by Equation \eqref{derdelta} and Lemma \ref{prel}.

\end{proof}
\begin{lemma}\label{dklladrs}
For all $k,l,r,s$ we have $[d_{kl}[\zb]_\la d_{rs}[\zb]]=[d_{kl},d_{rs}][\zb]$. \end{lemma}
\begin{proof}
We have by \eqref{cataliz}
\[
[\delta(\xb-\zb)d_{kl},\delta(\xb-\wb)d_{rs}]=\delta(\xb-\zb)\delta(\xb-\wb)[d_{kl},d_{rs}]=\delta(\wb-\zb)\delta(\xb-\wb)[d_{kl},d_{rs}].
\]
\end{proof}
\begin{lemma}\label{deiladjk}
We have
\begin{align*}
[\de_{x_i}[\zb],d_{jk}[\wb]]=&-(\de_{w_i}d_{jk}[\wb])\delta(\wb-\zb)-d_{jk}[\wb]\de_{w_i}\delta(\zb-\wb)\\
&+\delta_{ij}\sum_{h\neq k}d_{hk}[\wb]\de_{w_h}\delta(\zb-\wb)-\delta_{ki}\sum_{r\neq j}d_{rj}[\wb]\de_{w_r}\delta(\zb-\wb)
\end{align*}
and so
\[
[\de_{x_i}[\zb]_\la d_{jk}[\zb]]=-\de_{z_i}d_{jk}[\zb]-d_{jk}[\zb]\lambda_i+\delta_{ij}\sum_{h\neq k} d_{hk}[\zb]\lambda_h-\delta_{ki}\sum_{r\neq j}d_{rj}[\zb] \lambda_r.
\]
\end{lemma}
\begin{proof}
We have, by Lemma \ref{prel},
\begin{align*}
[\delta(\xb-\zb)& \de_{x_i}, \delta(\xb-\wb)d_{jk}]\\
&=\delta(\xb-\zb)(\de_{x_i}\delta(\xb-\wb))d_{jk}+\delta(\xb-\wb)\big(\delta_{ij}d(\delta(\xb-\zb)dx_k)-\delta_{ik}d(\delta(\xb-\zb)dx_j)\big)\\
&=-(\de_{w_i}\delta(\xb-\wb))d_{jk}\delta(\wb-\zb)-\delta(\xb-\wb)d_{jk}(\de_{w_i}\delta(\wb-\zb))\\
&\hspace{5mm}+\delta(\xb-\wb) \big( \delta_{ij} \sum_{h\neq k}(\de_{x_h}\delta(\xb-\zb))d_{hk}-\delta_{ik}\sum_{r\neq j} (\de_{x_r}\delta(\xb-\zb))d_{rj}\big),\end{align*}
and the statement follows by \eqref{cataliz}.
\end{proof}

Let us denote by $\mathcal Z$ the
$\mathbb F[\deb_{\zb}]$-module generated by the formal distributions $\de_{x_i}[\zb]$ and $d_{jk}[\zb]$.

\begin{lemma}\label{tildeRfree}
 The $\mathbb F[{\deb}_{\zb}]$-module $\mathcal Z$ is free of rank 15 on generators $\de_{x_i}[\zb]$ and $d_{jk}[\zb]$, $i,j,k\in \{1,2,3,4,5\}$, $j<k$.
 \end{lemma}
 
 \begin{proof}
 We first notice that if $P(\deb_{\zb})$ is any polynomial in $\mathbb F[\deb_{\zb}]$ such that $P(\deb_{\zb})\delta(\xb-\zb)=0$ then $P(\deb_{\zb})=0$. 
 Now let 
 \[
 \sum_{i}P_i(\deb_{\zb})\de_{x_i}[\zb]+\sum_{j<k}Q_{jk}(\deb_{\zb})d_{jk}[\zb]=0
 \]
 for some $P_i(\deb_{\zb}),Q_{jk}(\deb_{\zb})\in \mathbb F[\deb_{\zb}]$.
 Since the set $\{\xb^N \de_{x_i}:\, N\in \mathbb Z^5,\,i=1,\ldots,5\}\cup \{\xb^N d_{jk} :\, N\in \mathbb Z^5,\,j,k=1,\ldots,5,\, j<k\}$ is a linear basis of $A_5$, we have that  $P_i(\deb_{\zb})\delta(\xb-\zb)=0$ and $Q_{jk}(\deb_{\zb})\delta(\xb-\zb)=0$ and the result follows.
 \end{proof}

By Lemma  \ref{tildeRfree} we can identify the $\C[\deb]$-module $\widehat{RE}(5,10)$ with the $\C[\deb_{\zb}]$-module $\mathcal Z$. Moreover the $\la$-brackets on $\widehat{RE}(5,10)$ correspond to brackets of the corresponding formal distributions in $\mathcal Z$ via the formal Fourier transform thanks to Lemmas \ref{deiladej}, \ref{dklladrs} and \ref{deiladjk}. 
We restrict our attention to the following subspace of $\widehat{RE}(5,10)$:
\begin{definition}
 Let $RE(5,10)$ be the $\mathbb F[\deb]$-submodule of $\widehat{RE}(5,10)$ generated by the following elements:
 \begin{itemize}
  \item $a_{ij}=\de_{i}\de_{x_j}-\de_{j}\de_{x_i}$ for all $i,j\in \{1,2,3,4,5\}$;
  \item $b_k=\sum_h \de_{h}d_{hk}$ for all $k\in \{1,2,3,4,5\}$.
 \end{itemize}
\end{definition}
  One can also give the following abstract presentation of  $RE(5,10)$.

 \begin{proposition}\label{relations}
$RE(5,10)$ is  generated  as  a $\mathbb F[{\deb}]$-module by elements $a_{ij}, b_k$, where $i,j,k\in \{1,2,3,4,5\}$, subject to the following relations:
  \begin{enumerate}
  \item $a_{ij}+a_{ji}=0$;
  \item $\de_{h}a_{ij}+\de_{i}a_{jh}+\de_{j}a_{hi}=0$;
  \item $\sum_{k}\de_{k}b_k=0$.
  \end{enumerate}
 \end{proposition}
 
 \begin{proof}
 It is a simple verification that the generators $a_{ij}$ and $b_k$ satisfy the stated relations. By construction the elements $d_{jk}$ ($j<k$) and $\de_{x_i}$ are free generators of $\widehat{RE}[5,10]$. 
 Assume we have a relation
 \begin{equation}\label{pq}
 \sum_{k} P_k(\deb)b_k+\sum_{i,j} Q_{ij}(\deb)a_{ij}=0.
 \end{equation}
 Using relation (1), we can assume $Q_{ij}=-Q_{ji}$ and in particular $Q_{ii}=0$.
 Then we have
 \[
\sum_k \sum_{h\neq k} P_k(\deb) \de_{h}d_{hk}+\sum_{i,j}Q_{ij}(\deb)(\de_{i}\de_{x_j}-\de_{j}\de_{x_i})=0
\]
and so 
\[
\sum_{h<k} (P_k(\deb) \de_{h}-P_h(\deb)\de_{k})d_{hk}-\sum_{i}\big(\sum_j 2Q_{ij}(\deb)\de_{j}\big)\de_{x_i}=0.
\]
In particular we have for all $h\neq k$ that $P_h(\deb)\de_{k}=P_k(\deb)\de_{h}$ and hence there exists a polynomial $P(\deb)$ such that $P_k(\deb)=P(\deb)\de_{k}$ for all $k$, hence  the relation involving the $b_k$'s in \eqref{pq} is a consequence of relation (3). 

Using relation (2) systematically, we can assume that if $i<j$ the polynomial $Q_{ij}$ is actually a polynomial in the variables $\de_{h}$ with $h\leq j$. With this assumption we show that all polynomials $Q_{ij}$ vanish by induction on the lexicographic order of the pair of indices $(i,j)$.

If $(i,j)=(1,2)$ the relation $Q_{12}\de_{2}+Q_{13}\de_{3}+Q_{14}\de_{4}+Q_{15}\de_{5}=0$ implies $Q_{12}=0$ since $Q_{12}\in \mathbb F[\de_{1},\de_{2}]$. Let $(i,h)\neq (1,2)$ with $i<h$ be such that $Q_{rs}=0$ for all $(r,s)<(i,h)$. Relation above provides
\[
\sum_{j}Q_{ij}(\deb)\de_{j}=0
\]
which becomes \[
\sum_{j\geq h}Q_{ij}\de_{j}=0
\]
and since $Q_{ih}$ is not a polynomial in $\de_{j}$ for all $j>h$ we deduce that $Q_{ih}=0$.
 \end{proof}
In the next result we compute the $\la$-brackets among generators of $RE(5,10)$.
 \begin{proposition}\label{prodotti}
We have
 \begin{itemize} 
 \item $[{a_{ij} }_\la a_{rs}]=\lambda_j\lambda_r a_{si}+\lambda_i\lambda_r a_{js}+\lambda_i\lambda_s a_{rj}+\lambda_j\lambda_sa_{ir}+
\lambda_i\de_{z_j}a_{rs}+\lambda_j\de_{z_i}a_{sr}$;
\item[]
 \item $[{a_{ij}}_\la {b_k}]= (\lambda_i\de_{z_j}-\lambda_j\de_{z_i})b_k+(\delta_{ik}\lambda_j-\delta_{jk}\lambda_i)\sum_{r}\lambda_rb_r $;
\item[]
\item $[{b_i}_{\la} b_i]=0$;
\item[]
 \item if $i\neq j$, $[{b_i} _\la b_j]=\varepsilon_{ijhkl}(\lambda_ha_{kl}+\lambda_ka_{lh}+\lambda_la_{hk}$) where $\{i,j,h,k,l\}=\{1,2,3,4,5\}$.
  \end{itemize}
 \end{proposition}
\begin{proof} Using properties \eqref{defsupconf1}, \eqref{defsupconf2}, \eqref{defsupconf3} of Definition \ref{defsupconf} and Definition \ref{hatRE}, we have:
\begin{align*}
[\de_{i}\de_{x_j}{}_\la\de_{r}\de_{x_s}{}]& =-\lambda_i(\lambda_r+\de_{r})[\de_{x_j}{}_\la\de_{x_s}{}]\\
&=(\lambda_i\lambda_r\de_{j}+\lambda_i\lambda_j\lambda_r+\lambda_i\de_{r}\de_{j}+\lambda_i\lambda_j\de_{r})\de_{x_s}{}+
(\lambda_i\lambda_r\lambda_s+\lambda_i\lambda_s\de_{r})\de_{x_j}{}.
\end{align*}
Using four times this relation we obtain
\begin{align*}
 [{a_{ij}}_\la a_{rs}]&=[\de_{i}\de_{x_j}{}_\la\de_{r}\de_{x_s}{}]-[\de_{j}\de_{x_i}{}_\la\de_{r}\de_{x_s}{}]-[\de_{i}\de_{x_j}{}_\la\de_{s}\de_{x_r}{}]+[\de_{j}\de_{x_i}{}_\la\de_{s}\de_{x_r}{}]\\
&= (\lambda_i\lambda_r\de_{j}-\lambda_j\lambda_r\de_{i}+\lambda_i\de_{r}\de_{j}-\lambda_j\de_{r}\de_{i})\de_{x_s}{}+(\lambda_i\lambda_s \de_{r}-\lambda_i\lambda_r \de_{s})\de_{x_j}{}\\
&\hspace{5mm} +(-\lambda_i\lambda_s\de_{j}+\lambda_j\lambda_s \de_{i}-\lambda_i\de_{s}\de_{j}+\lambda_j \de_{s}\de_{i})\de_{x_r}{} +(-\lambda_j\lambda_s \de_{r}+\lambda_j\lambda_r \de_{s})\de_{x_i}{}\\
&=\lambda_j\lambda_r a_{si}+\lambda_i\lambda_r a_{js}+\lambda_i\lambda_s a_{rj}+\lambda_j\lambda_sa_{ir}+
\lambda_i\de_{j}a_{rs}+\lambda_j\de_{i}a_{sr}.
\end{align*}

In order to compute $[{a_{ij}}_\la b_k]$ we first assume that $i,j,k$ are distinct. We have
\begin{align*}
 [{a_{ij}} _\la b_k]&=\Big[(\de_{i}\de_{x_j}-\de_{j}\de_{x_i})_\la \sum_{h\neq k} \de_{h}d_{hk}\Big]\\
 &=-\lambda_i \sum_{h\neq k}(\de_{h}+\lambda_h)[{\de_{x_j}}_\la d_{hk}]+\lambda_j \sum_{h\neq k} (\de_{h}+\lambda_h)[{\de_{x_i} }_\la d_{hk}]\\
 &= \lambda_i \sum_{h\neq k}(\de_{h}+\lambda_h)(\de_{j}+\lambda_j)d_{hk}-\lambda_i(\de_{j}+\lambda_j)\sum_{r\neq k}\lambda_rd_{rk}\\
 &-\lambda_j\sum_{h\neq k} (\de_{h}+\lambda_h)(\de_{i}+\lambda_i) d_{hk}+\lambda_j(\de_{i}+\lambda_i)\sum_{r\neq k}\lambda_r d_{rk}\\
 &= \lambda_i(\de_{j}+\lambda_j)\sum_{h\neq k}\de_{h}d_{hk}-\lambda_j(\de_{i}+\lambda_i)\sum_{h\neq k}\de_{h}d_{hk}\\
 &= (\lambda_i\de_{j}-\lambda_j\de_{i}) b_k.
 \end{align*}

Now we assume that $i,j,k$ are not distinct and with no loss of generality we can assume that $k=i\neq j$. We have

\begin{align*}
 [{a_{ij}} _\la b_i]&=\Big[(\de_{i}\de_{x_j}-\de_{j}\de_{x_i})_\la \sum_{h\neq i} \de_{h}d_{hi}\Big]\\
 &=-\lambda_i \sum_{h\neq i}(\de_{h}+\lambda_h)[{\de_{x_j}}_\la d_{hi}]+\lambda_j \sum_{h\neq i} (\de_{h}+\lambda_h)[{\de_{x_i} }_\la d_{hi}]\\
 &= \lambda_i\sum_{h\neq i}(\de_{h}+\lambda_h)(\de_{j}+\lambda_j)d_{hi}-\lambda_i(\de_{j}+\lambda_j)\sum_{r\neq i}\lambda_rd_{ri}\\&\hspace{5mm} -\lambda_j\sum_{h\neq i}(\de_{h}+\lambda_h)\big((\de_{i}+\lambda_i) d_{hi}+\sum_{s\neq h}\lambda_s d_{sh}\big)\\
 &= \lambda_i(\de_{j}+\lambda_j) \sum_{h\neq i}\de_{h}d_{hi}-\lambda_j(\de_{i}+\lambda_i) \sum_{h\neq i}(\de_{h}+\lambda_h)d_{hi} -\lambda_j \sum_{h\neq i}\sum_{s\neq h}(\de_{h}+\lambda_h)\lambda_s d_{sh}\\
 &= \lambda_i(\de_{j}+\lambda_j)b_i-\lambda_j(\de_{i}+\lambda_i)b_i-\lambda_j(\de_{i}+\lambda_i)\sum_{h\neq i}\lambda_hd_{hi}\\&\hspace{5mm}-\lambda_j \sum_{h\neq i}\sum_{s\neq h}\de_{h}\lambda_s d_{sh}-\lambda_j\sum_{h\neq i}\sum_{s\neq h}\lambda_h\lambda_s d_{sh}\\
 \end{align*}
In the last sum all terms with $s\neq i$ cancel out and so we have
\begin{align*}
 [{a_{ij}} _\la b_i]&=(\lambda_i\de_{j}-\lambda_j \de_{i})b_i-\lambda_j\de_{i}\sum_{h\neq i}\lambda_h d_{hi}-\lambda_i\lambda_j\sum_{h\neq i}\lambda_h d_{hi}\\&\hspace{5mm}-\lambda_j\sum_s\lambda_s \sum_{h\neq i,s}\de_{h}d_{sh}-\lambda_j \sum_{h\neq i}\lambda_h \lambda_i d_{ih}\\
 &=(\lambda_i\de_{j}-\lambda_j \de_{i})b_i-\lambda_j\sum_{s}\lambda_s\sum_{h\neq s}\de_hd_{sh}\\
 &= (\lambda_i\de_{j}-\lambda_j \de_{i})b_i+\lambda_j\sum_{s}\lambda_s b_s.
\end{align*}

Now we compute
\begin{align*}
[{b_i}_\la b_i]&=\Big[\sum_{h\neq i} {\de_{h}d_{hi}}_\la \sum_{k\neq i} \de_{k}d_{ki}\Big]\\
&=-\sum_{h,k}\lambda_h(\lambda_k+\de_{k})[{d_{hi}}_\la d_{ki}]=0.
\end{align*}
 
Now we compute $[{b_i}_\la b_j]$ for $i\neq j$. Let $h,k,l$ be such that $\varepsilon_{ijhkl}=1$. We have:
\begin{align*}
[{b_i}_\la b_j]&=\Big[\sum_{r\neq i} {\de_{r}d_{ri}}\,_\la \sum_{s\neq j} \de_{s}d_{sj}\Big]\\
&=-\sum_{r,s}\lambda_r(\lambda_s+\de_{s})[{d_{ri}}_\la d_{sj}]\\
&=(\lambda_h\lambda_k+\lambda_h\de_{k}-\lambda_k\lambda_h-\lambda_k\de_{h})\de_{x_l}+
(\lambda_l\lambda_h+\lambda_l\de_{h}-\lambda_h\lambda_l-\lambda_h\de_{l})\de_{x_k}\\
&\hspace{5mm}
+(\lambda_k\lambda_l+\lambda_k\de_{l}-\lambda_l\lambda_k-\lambda_l\de_{k})\de_{x_h}\\
&=\lambda_h(\de_{k}\de_{x_l}-\de_{l}\de_{x_k})+\lambda_k(\de_{l}\de_{x_h}-\de_{h}\de_{x_l})
+\lambda_l(\de_{h}\de_{x_k}-\de_{k}\de_{x_h}).
\end{align*}
\end{proof}

\begin{theorem}
The $\C[\deb]$-module $RE(5,10)$ with the $\la$-bracket induced from $\widehat{RE}(5,10)$ and given in Proposition \ref{prodotti} 
is a Lie conformal superalgebra of type $(5,0)$.
\end{theorem}
\begin{proof}
By Proposition \ref{prodotti} 
the $\la$-bracket actually restricts to a linear map
\[
[\,.\,_\la\,.\,]:RE(5,10)\times RE(5,10) \rightarrow \mathbb F[\la]\otimes RE(5,10)
\]
satisfying conformal sesquilinearity and conformal skew-symmetry. We need to prove the conformal Jacoby identity
for the triple $(a,b,c)$ in the following four cases
\begin{enumerate}
\item $(a,b,c)=(a_{ij},a_{rs}, a_{mn})$;
\item $(a,b,c)=(b_i,b_j,b_k)$;
\item $(a,b,c)=(a_{ij},a_{rs}, b_k)$;
\item $(a,b,c)=(a_{ij}, b_h,b_k)$.
\end{enumerate}

(1) Since the elements $a_{ij}$ belong to the submodule of $\widehat{RE}(5,10)$ generated by the elements $\de_{x_i}$'s it is enough to consider the Jacobi identity for elements $\de_{x_i},\de_{x_j}, \de_{x_k}$ and this can be verified as in Proposition \ref{4.2}.

(2) We have to show that
\[
[b_i {}_\la[b_j {}_\mub b_k]]=[[b_i{}_\la b_j]_{\la+\mub}b_k]-[b_j{}_\mub [b_i{}_\la b_k]].
\]
For distinct $i,j,k$ let $r,s$ be such that $\varepsilon_{ijkrs}=1$. We have:
\begin{align*}
[{b_i}_\la[{b_j}_\mub b_k]]&=[{b_i}_\la (\mu_ia_{rs}+\mu_ra_{si}+\mu_sa_{ir})]\\
&=-\mu_i(-\lambda_r\de_{s}+\lambda_s\de_{r})b_i+\mu_r(-\lambda_i\de_{s}b_i+\lambda_s\de_{i}b_i+
(\lambda_s+\de_{s})\sum_h(\lambda_h+\de_{h})b_h)\\
&\hspace{5mm}-\mu_s(-\lambda_i\de_{r}b_i+\lambda_r\de_{i}b_i+
(\lambda_r+\de_{r})\sum_h(\lambda_h+\de_{h})b_h)\\
&=(\lambda_s\mu_r\de_{i}-\lambda_i\mu_r\de_{s}+\lambda_i\mu_s\de_{r}-\lambda_r\mu_s\de_{i}+\lambda_r\mu_i\de_{s}-\lambda_s\mu_i\de_{r})b_i\\
&\hspace{5mm}+(\lambda_s\mu_r+\mu_r\de_{s}-\lambda_r\mu_s-\mu_s\de_{r})\sum_{h}\lambda_hb_h.
\end{align*}

Similarly we can compute
\begin{align*}
[[b_i{}_\la b_j]_{\la+\mub}b_k]&=(\lambda_k\mu_r\de_{s}-\lambda_k\mu_s\de_{r}+\lambda_r\mu_s\de_{k}-\lambda_r\mu_k\de_{s}+\lambda_s\mu_k\de_{r}-\lambda_s\mu_r\de_{k})b_k\\
&\hspace{5mm}+(\lambda_s\mu_r-\lambda_r\mu_s)\sum_{h}(\lambda_h+\mu_h)b_h
\end{align*}
and
\begin{align*}
[b_j{}_\mub [b_i{}_\la  b_k]&=-\Big(\big(\lambda_r\mu_s\de_{j}-\lambda_r\mu_j\de_{s}+\lambda_s\mu_j\de_{r}-\lambda_s\mu_r\de_{j}+\lambda_j\mu_r\de_{s}-\lambda_j\mu_s\de_{r}\big)b_j\\
&\hspace{5mm}+\big(\lambda_r\mu_s+\lambda_r\de_{s}-\lambda_s\mu_r-\lambda_s\de_{r}\big)\sum_h \mu_h b_h\Big).
\end{align*}
So we have

\begin{align*}
[&b_i {}_\la[b_j {}_\mub b_k]]-[[b_i{}_\la b_j]_{\la+\mub}b_k]+[b_j{}_\mub [b_i{}_\la b_k]]\\
&=(\lambda_s\mu_r\de_{i}-\lambda_i\mu_r\de_{s}+\lambda_i\mu_s\de_{r}-\lambda_r\mu_s\de_{i}+\lambda_r\mu_i\de_{s}-\lambda_s\mu_i\de_{r})b_i\\
&\hspace{5mm}+(-\lambda_k\mu_r\de_{s}+\lambda_k\mu_s\de_{r}-\lambda_r\mu_s\de_{k}+\lambda_r\mu_k\de_{s}-\lambda_s\mu_k\de_{r}+\lambda_s\mu_r\de_{k})b_k\\
&\hspace{5mm} +\big(-\lambda_r\mu_s\de_{j}+\lambda_r\mu_j\de_{s}-\lambda_s\mu_j\de_{r}+\lambda_s\mu_r\de_{j}-\lambda_j\mu_r\de_{s}+\lambda_j\mu_s\de_{r}\big)b_j\\
&\hspace{5mm}+(\mu_r\de_{s}-\mu_s \de_{r})\sum_h \lambda_h b_h+ (-\lambda_r\de_{s}+\lambda_s \de_{r})\sum_h \mu_h b_h\\
&=(\lambda_s\mu_r\de_{i}-\lambda_r\mu_s \de_{i})b_i+(\lambda_s\mu_r\de_{k}-\lambda_r\mu_s\de_{k})b_k+(\lambda_s\mu_r\de_{j}-\lambda_r\mu_s\de_{j})b_j\\
&\hspace{5mm} +(\lambda_s\mu_r\de_{r}-\lambda_r\mu_s\de_{r})b_r+(\lambda_s\mu_r\de_{s}-\lambda_r\mu_s\de_{s})b_s\\
&=(\lambda_s\mu_r-\lambda_r\mu_s)\sum_h\de_{h}b_h\\
&=0.
\end{align*}

(3) First assume $i,j,r,s,k$ are distinct.
We have
\begin{align*}
[a_{ij}{}_\la [a_{rs}{}_\mub b_k]]&=[a_{ij}{}_\la (\mu_r\de_{s}-\mu_s \de_{r})b_k]\\
&=\big((\mu_r(\de_{s}+\lambda_s)-\mu_s(\de_{r}+\lambda_r))(\lambda_i \de_{j}-\lambda_j \de_{i})\big)b_k\\
&=\big((\mu_r\de_{s}-\mu_s\de_{r})-(\lambda_r\mu_s-\lambda_s\mu_r)\big)(\lambda_i \de_{j}-\lambda_j \de_{i})\big)b_k
\end{align*}
and
\begin{align*}
[[a_{ij}{}_\la a_{rs}]{}_{\la+\mub}b_k]&=\big[\big(\lambda_j\lambda_r a_{si}+\lambda_i\lambda_r a_{js}+\lambda_i\lambda_s a_{rj}+\lambda_j\lambda_s a_{ir}+(\lambda_i\de_{j}-\lambda_j\de_{i})a_{rs}\big)_{\la+\mub}b_k\big]\\
&=\big(\lambda_j\lambda_r((\lambda_s+\mu_s)\de_{i}-(\lambda_i+\mu_i)\de_{s})+\lambda_i\lambda_r((\lambda_j+\mu_j)\de_{s}-(\lambda_s+\mu_s)\de_{j})\\
&\hspace{5mm}+\lambda_i\lambda_s((\lambda_r+\mu_r)\de_{j}-(\lambda_j+\mu_j)\de_{r})+\lambda_j\lambda_s((\lambda_i+\mu_i)\de_{r}-(\lambda_r+\mu_r)\de_{i})\\
&\hspace{5mm}+(\lambda_i(-\lambda_j-\mu_j)-\lambda_j(-\lambda_i-\mu_i))((\lambda_r+\mu_r)\de_{s}-(\lambda_s+\mu_s)\de_{r})\big)b_k\\
&= -(\lambda_i\de_{j}-\lambda_j\de_{i})(\lambda_r\mu_s-\lambda_s\mu_r)-(\lambda_i\mu_j-\lambda_j\mu_i)(\mu_r\de_{s}-\mu_s\de_{r})
\end{align*}
and finally
\[
[a_{rs}{}_\mub [a_{ij}{}_\mub b_k]]=\big((\lambda_i\de_{j}-\lambda_j\de_{i})-(\mu_i\lambda_j-\mu_j\lambda_i)\big)(\mu_r \de_{s}-\mu_s \de_{r})\big)b_k
\]
and so the conformal Jacobi identity follows also in this case.

The other cases can be carried out similarly.
\end{proof}
We observe that $RE(5,10)$ is a $\Z$-graded Lie conformal superalgebra if we  set
$$\deg(a_{ij})=-4,\,\,\, \deg(b_k)=-3.$$
\begin{proposition}\label{dimE510}
Let $k\geq 0$. Then
\[
\dim \mathfrak{g}(RE(5,10))_{2k-4}\leq\frac{1}{6}k(k+1)(k+2)(k+4)
\]
and
\[
\dim \mathfrak{g}(RE(5,10))_{2k-3}\leq\frac{1}{6}k(k+2)(k+3)(k+4).
\] 
\end{proposition}
\begin{proof}
Consider the elements $a_{ij}\yb ^R$ with $|R|=k$ which generate  $\mathfrak{g}(RE(5,10))_{2k-4}$. Since $a_{ij}=-a_{ji}$ we can always assume that $i<j$. We show that the set 
\[\mathcal B=\{a_{i\,i+1}\yb^{M+e_i} : \,|M|=k-1,\,i=1,\ldots,4\}\cup \{a_{12}\yb^{N+e_2}: \, |N|=k-1,\, y_1\not | \yb^N\}\] 
spans $\mathfrak{g}(RE(5,10))_{2k-4}$.
 
To show this we choose a total order $\preceq$ on the monomials $a_{ij}\yb ^R$ with $i<j$ such that 
\begin{itemize}
\item if $h\leq i<j\leq k$ then $a_{ij}\yb ^R\preceq a_{hk} \yb^S$ for all $R,S$.
\item if $i<j<k$ then $a_{ij}\yb ^R\prec a_{jk} \yb^S$ for all $R,S$.
\end{itemize}
We prove that for every element $a_{ij}\yb ^R$ that is not in $\mathcal B$ there exists a relation
\begin{equation}\label{evre}
a_{ij} \de_{y_h}\yb^M+a_{jh} \de_{y_i}\yb^M+ a_{hi}\de_{y_j}\yb^M=0,
\end{equation}
which expresses it as  a linear combination of smaller elements. 

Let $a_{ij}\yb^R\notin B$. 
If $j-i>2$ we let $i<h<j$ and  we have
\[a_{ij}\yb^R=\frac{1}{r_h+1}a_{ij}\de_{y_h} y_h \yb^R=\frac{1}{r_h+1} \big(a_{ih}\de_{y_j}(y_h\yb^R)-a_{jh}\de_{y_i} (y_h \yb ^R)  \big).\]

Let us consider the case $j=i+1$. If $y_i,y_{i+1}$ do not divide $\yb^R$ we have 
\[a_{i\, i+1} \yb ^R=\frac{1}{r_h+1} \big(a_{ih}\de_{y_{i+1}}(y_h\yb^R)-a_{i+1\,h}\de_{y_i} (y_h \yb ^R)  \big)=0,\]
where $h$ is any index distinct from $i$, $i+1$. 
So we can assume that $i>1$ and that $y_{i+1}$ is a divisor of  $\yb^R$ but $y_{i}$ is not (otherwise $a_{i\,i+1}\yb^R\in \mathcal B$). We have 
\[
a_{i\, i+1} \yb ^R=\frac{1}{r_{i-1}+1} \big(a_{i\, \,i-1}\de_{y_{i+1}}(y_{i-1}\yb^R)-a_{i+1\,i-1}\de_{y_i} (y_{i-1} \yb ^R)  \big)=-\frac{1}{r_{i-1}+1} a_{i-1\,i}\de_{y_{i+1}}(y_{i-1}\yb^R).
\]
So the  dimension of $\mathfrak{g}(RE(5,10))_{2k-4}$ is less than or equal to the cardinality of $\mathcal B$, i.e.
\[
4\binom{5+k-2}{4}+\binom{4+k-2}{4}=\frac{1}{6}k(k+1)(k+2)(k+4).
\]
One can similarly show that the set
\[
\mathcal B '=\{b_i\yb^M:\, y_j|\yb^M \textrm{ for some }j>i, |M|=k\}
\]
is a generating set of $\mathfrak{g}(RE(5,10))_{2k-3}$. The proof is analogous and simpler than in the former case and is based on the observation that if $b_i \yb^M \notin \mathcal B'$  then
\[b_i \yb^M=\frac{1}{m_i+1}\de_{y_i}y_i \yb^M=-\frac{1}{m_i+1}\sum_{j\neq i}b_j\de_{y_j} y_i \yb^M=-\frac{1}{m_i+1}\sum_{j<i}b_j\de_{y_j} y_i \yb^M.
\]
So the dimension of $\mathfrak{g}((5,10))_{2k-3}$ is at most
\[
4\binom{k+3}{4}+3\binom{k+2}{3}+2\binom{k+1}{2}+\binom{k}{1}=\frac{1}{6}k(k+2)(k+3)(k+4).
\]

\end{proof}
\begin{corollary}\label{eqdim}
For all $d\in \mathbb Z$
\[
\dim E(5,10)_d\geq \dim \mathfrak g(RE(5,10))_d.
\]
\end{corollary}
\begin{proof}
Using that all homogeneous components $E(5,10)_d$ are irreducible $E(5,10)_0=\mathfrak{sl}_5$-modules with
known highest weights (\cite{CK2}),
one can check that the dimension
 of  $E(5,10)_{2k-4}$ is  
 $\frac{1}{6}k(k+1)(k+2)(k+4)$ and 
the dimension of  $E(5,10)_{2k-3}$ is $\frac{1}{6}k(k+2)(k+3)(k+4)$. Hence corollary follows from Proposition \ref{dimE510}.
\end{proof}
\begin{theorem}\label{E(5,10)} The annihilation algebra $\mathfrak{g}(RE(5,10))$  is isomorphic, as a $\Z$-graded
Lie superalgebra, to $E(5,10)$ with principal gradation under the following map $\Phi$:
$$a_{ij}\yb^M\mapsto -m_i\xb^{M-e_i}\de_{x_j}+m_j\xb^{M-e_j}\de_{x_i};$$
$$b_k\yb^M\mapsto \sum_{r\neq k}m_r\xb^{M-e_r}d_{kr}.$$
\end{theorem}
\begin{proof} Let $R=RE(5,10)$.
By construction and  Proposition \ref{relations}, ${\mathfrak g}(R)$ is spanned by the elements $a_{ij}\yb^M$ and $b_k\yb^M$
subject to the following relations:
\begin{itemize}
\item[(i)] $a_{ij}\yb^M+a_{ji}\yb^M=0$;
\item[(ii)] $a_{ij}{\de}_{y_h}\yb^M+a_{jh}{\de}_{y_i}\yb^M+a_{hi}{\de}_{y_j}\yb^M=0$;
\item[(iii)] $\sum_k b_k{\de}_{y_k}\yb^M=0$.
\end{itemize}
In particular $a_{ij}\yb^{0}=0$ and $b_k\yb^{0}=0$.
It is easy to check that relations (i), (ii), (iii) are preserved by the map $\Phi$ and that $\Phi$ is surjective. Injectivity of $\Phi$ follows from Corollary \ref{eqdim}.

The proof that $\Phi$ is an homomorphism is a straightforward verification based on the following observation
\begin{align*} 
[a_{ij}\yb^M, a_{rs}\yb^N]&=
a_{si}(\de_{y_j}\de_{y_r}\yb^M)\yb^{N}+a_{js}(\de_{y_i}\de_{y_r}\yb^M)\yb^{N}+a_{rj}(\de_{y_i}\de_{y_s}\yb^M)\yb^{N}\\
&\hspace{5mm}+
a_{ir}(\de_{y_j}\de_{y_s}\yb^M)\yb^{N}-a_{rs}{\de}_{y_j}((\de_{y_i}\yb^{M})\yb^N)-a_{sr}{\de}_{y_i}((\de_{y_j}\yb^{M})\yb^N).
\end{align*} 
\end{proof}
The following corollary answers a question 
raised in \cite{CC}.
\begin{corollary}
 Let $F$ be a finite-dimensional $\mathfrak{sl}_5$-module and let $M(F)$ be the corresponding
 $E(5,10)$-Verma module. Then $M(F)^*$ is isomorphic to $M(F^*)$.
\end{corollary}
\begin{proof}
By Theorem \ref{E(5,10)} we have $\mathfrak{g}(RE(5,10))_{-2}=\langle \de_{y_1},\dots, \de_{y_5}\rangle$ and $RE(5,10)$ satisfies Assumptions \ref{assumption}. Since the $0$-th degree component of $E(5,10)$ is isomorphic
to $\mathfrak{sl}_5$ which has only zero characters, by Theorem \ref{main}  the conformal module $M(F)^*$ is isomorphic to $M(F^*)$. 
\end{proof}

\section{The Lie conformal superalgebra  $RE(3,6)$.}\label{section7}
In this section we introduce the infinite-dimensional (linearly compact) Lie superalgebra $E(3,6)$ and realize it
as the annihilation superalgebra of a Lie conformal superalgebra of type $(3,0)$. This is used to show that $M(F)^*=M(F^*)$ for every finite-dimensional representation $F$ of $E(3,6)_0$.

Recall the construction of $E(3,6)$ in \cite{CK2}. The even part of $E(3,6)$ is the semidirect sum of $W(3,0)$ and $\Omega^0(3)\otimes \sl_2$, where $\Omega^0(3)$ denotes the space of formal power series in three even indeterminates $x_1, x_2, x_3$ and  $W(3,0)$ acts on it in the standard way.
Besides, for $f,g\in \Omega^0(3)$ and $c_1, c_2\in \sl_2$ we have:
$$[f\otimes c_1, g\otimes c_2]=fg\otimes [c_1,c_2].$$
The odd part of $E(3,6)$  is $\Omega^1(3)^{-1/2}\otimes \C^2$ where  $\Omega^1(3)^{-1/2}$ is the space of 1-forms in $x_1, x_2, x_3$ and $-1/2$ refers to the action of $W(3,0)$ on $\Omega^1(3)$; namely, for $X\in W(3,0)$ and $\omega\otimes v\in \Omega^1(3)^{-1/2}\otimes \C^2$ we have:
$$[X, \omega\otimes v]={\mathcal L}_X(\omega)\otimes v-\frac{1}{2} \diver(X)\omega\otimes v$$
where ${\mathcal L}_X$ denotes the Lie derivative and $\diver$ the usual divergence. 
Besides, for $f\in \Omega^0(3)$ and $c\in \sl_2$, we have:
\[
 [f\otimes c, \omega\otimes v]=f\omega\otimes c.v
\]
where $c.v$ denotes the standard action of $c$ on $v$.
The bracket between two odd elements is defined as follows: we identify $\Omega^2(3)^{-1}$ with $W(3,0)$ via contraction of vector fields with the standard volume form, and $\Omega^3(3)^{-1}$ with $\Omega^0(3)$. Then, for $\omega_1, \omega_2\in \Omega^1(3)^{-\frac{1}{2}}$, $u_1, u_2\in\C^2$, we have:
\[
[\omega_1\otimes u_1, \omega_2\otimes u_2]=(\omega_1\wedge \omega_2)\otimes (u_1\wedge u_2)+
\frac{1}{2}(d\omega_1\wedge\omega_2+\omega_1\wedge d\omega_2)\otimes u_1\cdot u_2
\]
where $u_1\cdot u_2$ denotes an element in the symmetric square of $\C^2$, i.e., an element in $\mathfrak{sl}_2$, and $u_1\wedge u_2$ an element in the skew-symmetric square of $\C^2$, i.e., a complex number. Let us denote by $H,E,F$ the standard basis of $\sl_2$ and by $e_1, e_2$ the standard basis of $\C^2$. Then $E={e_1}^2/2$, $F=-{e_2}^2/2$, $H=-e_1\cdot e_2$ and $e_1\wedge e_2=1$.

For $i,j\in \{1,2,3\}$, $\varepsilon_{ij}$ and $t_{ij}$ are defined as follows: if $i\neq j $ we let $t_{ij}\in \{1,2,3\}$ be such that $|\{i,j,t_{ij}\}|=3$ and $\varepsilon_{ij}$ be the sign of the permutation 
$(i,j,t_{ij})$. If $i=j$ we let $\varepsilon_{ij}=0$.

We consider on $E(3,6)$ the principal gradation given by $\deg x_i=-\deg \de_{x_i}=2$, $\deg(e_1)=\deg(e_2)=0$, $\deg(E)=\deg(F)=\deg(H)=0$, $\deg(dx_i)=-1$ for all $i=1,2,3$. Observe that the $0$-th graded component is isomorphic to $\mathfrak{sl}_3\oplus \mathfrak{sl}_2\oplus \mathbb F$, $E(3,6)_{-1}=\langle dx_i\otimes e_j:\, i=1,2,3,\,j=1,2\rangle$ and $E(3,6)_{-2}=\langle \de_{x_1},\de_{x_2}, \de_{x_3}\rangle$.

The strategy to construct the Lie conformal superalgebra $R\mathfrak{g}$ for 
$\mathfrak{g}=E(3,6)$ 
and $E(3,8)$ (see Section \ref{section8}), such that the annihilation Lie superalgebra of $R\mathfrak{g}$ is $\mathfrak{g}$, is 
the same as for $\mathfrak{g}=E(5,10)$ in Section \ref{section6} (see also the Introduction). 
Namely, we construct a formal 
distribution Lie superalgebra $\tilde{\mathfrak{g}}$ by localizing the
formal power series in $x_1, x_2, x_3$ by these 
variables, and show that $\tilde{\mathfrak{g}}$ is spanned by the coefficients
of pairwise local formal distributions $a_i(\zb)$.
Next, we compute the brackets 
$[a_i(\zb),a_j(\wb)]$ as linear combinations of the delta function and its 
derivatives with coefficients the formal distributions $a_k(\wb)$ and their 
derivatives. Applying the formal Fourier transform, we obtain the Lie conformal 
superalgebra $R\mathfrak{g}$ of type $(3,0)$. The annihilation Lie 
superalgebra of $R\mathfrak{g}$ has a canonical surjective map  $\Phi$ to 
$\mathfrak{g}$, and it 
remains to show that the kernel of $\Phi$ is zero. In the case when all the 
coefficients of all the formal distributions $a_i(\zb)$ are linearly independent, like for $\mathfrak{g}=E(3,6)$,
this is immediate. But for $E(5,10)$ and $E(3,8)$ this does not hold,
and it requires some effort to prove that $\Phi$ has zero kernel.
(The case of Lie conformal superalgebras of type $(1,0)$ is easy since
any finitely generated module over $\C[\deb]$ is a direct sum of a free 
module and a torsion module.)

So let $\mathfrak g=E(3,6)$ and $\tilde{\mathfrak g}$ be defined as above. In order to construct the conformal superalgebra $RE(3,6)$ we consider the following formal distributions in the variable $z$ with coefficients in $\tilde{\mathfrak g}$:
$a_i(\zb)=\delta(\xb-\zb)\de_{x_i}$ , $c(\zb)=\delta(\xb-\zb) \otimes c$, and $b_{hk}(\zb)=\delta(\xb-\zb)  dx_h\otimes e_k$, for all $i,h=1,2,3$, $k=1,2$, $c\in \mathfrak{sl}_2$. The computation of the formal Fourier transform of the brackets between these formal distributions leads us to the following definition.
\begin{definition}\label{def:re36}
	We set $RE(3,6)=\C[\de_1,\de_2,\de_3](A\oplus B\oplus C)$, where $A=\langle a_1, a_2, a_3\rangle$
	is even of degree $-2$, 
	$B=\langle b_1, b_2, b_3\rangle\otimes\C^2$ is odd of degree $-1$ and $C=\sl_2$ is even of degree zero,
	 and the following $\la$-brackets: for $c,c'\in C$, $v\in\C^2$ and $b_{hk}=b_h\otimes e_k$,
	 \begin{align*}
	&[{a_i}_\la a_j]=(\de_i+\lambda_i)a_j+\lambda_ja_i\\
	&[{a_i}_\la b_{hk}]=(\de_i+\frac{3}{2}\lambda_i)b_{hk}-\delta_{ih}(\lambda_1b_{1k}+\lambda_2b_{2k}+\lambda_3b_{3k})\\
	& [{b_{ij}}_{\la}b_{hk}]=(j-k)\varepsilon_{ih}a_{t_{ih}}\\
	& [{a_i}_{\la}c]=(\de_i+\lambda_i)c\\
	& [c_{\la}b_h\otimes v]=b_h\otimes c.v\\
	&[c_{\la}c']=[c,c'].
	\end{align*}
	extended on the whole $RE(3,6)$ by properties \eqref{defsupconf1} and \eqref{defsupconf2} of Definition \ref{defsupconf}.
\end{definition}
\begin{proposition}
	$RE(3,6)$ is a Lie conformal superalgebra of type $(3,0)$.
\end{proposition}
\begin{proof}
The only nontrivial verification is the conformal Jacobi identity which can be verified directly if one considers Definition \ref{def:re36} as an abstract definition, or as a consequence of the fact that we constructed $RE(3,6)$ as a space of formal distributions with coefficients in the Lie superalgebra $\tilde {\mathfrak g}$. 
\end{proof}
\begin{theorem}
 The annihilation algebra ${\mathfrak g}(RE(3,6))$ is isomorphic, as $\Z$-graded Lie superalgebra, to $E(3,6)$ with principal gradation via the following map:
 $$
 \begin{array}{lcl}
 a_i\yb^M &\mapsto& -\xb^M\de_{x_i};\\
 b_{hk}\yb^M& \mapsto &\xb^M dx_h\otimes e_k;\\
 c\yb^M &\mapsto &\xb^M \otimes c.
 \end{array}$$
\end{theorem}
\begin{proof}
 The map is clearly a linear isomorphism. We need to check that the Lie bracket is preserved. We have:
 \begin{align*}
 [a_i\yb^M, b_{hk}\yb^N]&=\de_ib_{hk}\yb^{M+N}+\frac{3}{2}m_ib_{hk}\yb^{M+N-e_i}-\delta_{ih}\sum_{j=1}^3m_jb_{jk}\yb^{M+N-e_j}\\
 &=(-n_i+\frac{1}{2}m_i)b_{hk}\yb^{M+N-e_i}-\delta_{ih}\sum_{j=1}^3m_jb_{jk}\yb^{M+N-e_j}.
\end{align*}
\begin{align*}
 [\xb^M\de_{x_i}, \xb^Ndx_h\otimes e_k]&=n_i\xb^{M+N-e_i}dx_h\otimes e_k+\xb^N(\delta_{ih}d(\xb^M)\otimes e_k
 -\frac{1}{2}m_i\xb^{M-e_i}dx_h\otimes e_k)\\
 &=(n_i-\frac{1}{2}m_i)\xb^{M+N-e_i}dx_{h}\otimes e_k+\delta_{ih}\sum_{j=1}^3m_j\xb^{M+N-e_j}dx_{j}\otimes e_k
\end{align*}

 The verification in the other cases is simpler and left to the reader.
\end{proof}

\begin{remark}Let $R=RE(3,6)$. Observe that $R$ satisfies Assumption \ref{assumption} and so Theorem \ref{main} applies.
Note that  $\mathfrak g(R)_0\cong \mathfrak{sl}_3\oplus \mathfrak{sl}_2\oplus \C z$
where $z$ is its non-zero central element such that $\ad z$ acts as $j$ on $\mathfrak g(R)_j$. Since $\dim\mathfrak g(R)_{-1}=6$, $\dim\mathfrak g(R)_{-2}=3$ and the depth of the gradation is $2$, we see that $\str \big(\ad z_{|\mathfrak g(R)_{<0}}\big)=0$ and this is consistent with the classification of degenerate Verma modules of $E(3,6)$ given in \cite{KR1}, \cite{KR2} and \cite{KR3}. 
\end{remark}
\section{The Lie conformal superalgebra $RE(3,8)$}\label{section8}
The Lie superalgebra
 $E(3,8)$ has the following structure \cite{CK2, CCK, CaK}(\cite{CCK} corrected \cite{CK2}, and \cite{CaK} further corrected \cite{CCK}): it has even part $E(3,8)_{\bar{0}}=W(3,0)+\Omega^0(3)\otimes \mathfrak{sl}_2+d\Omega^1(3)$
and odd part $E(3,8)_{\bar{1}}=\Omega^0(3)^{-1/2}\otimes \C^2+\Omega^2(3)^{-1/2}\otimes\C^2$. $W(3,0)$ acts on $\Omega^0(3)\otimes \mathfrak{sl}_2$ in the natural way and on $d\Omega^1(3)$ via Lie derivative.
For $X,Y\in W(3,0)$, $f,g\in\Omega^0(3)$, $A,B\in \mathfrak{sl}_2$, $\omega_1, \omega_2\in d\Omega^1(3)$, we have:

\begin{align*}
&[X,Y]=XY-YX-\frac{1}{2}d(\diver(X))\wedge d(\diver(Y)),\\
&[f\otimes A,\omega_1]=0,\\
&[f\otimes A, g\otimes B]=fg\otimes [A,B]+\tr(AB) df\wedge dg,\\
&[\omega_1,\omega_2]=0.
\end{align*}
Besides, for $X\in W_3$, $f\in\Omega^0(3)^{-1/2}$, $g\in\Omega^0(3)$, $v\in\C^2$, $A\in \mathfrak{sl_2}$, $\omega\in d\Omega^1(3)$,
$\sigma\in\Omega^2(3)^{-1/2}$,
\begin{align*}
&[X,f\otimes v]=(X.f+\frac{1}{2}d(\diver(X))\wedge df)\otimes v,\\
&[g\otimes A,f\otimes v]=(gf+dg\wedge df)\otimes Av,\\
&[g\otimes A, \sigma\otimes v]=g\sigma\otimes Av,\\
&[\omega,f\otimes v]=f\omega\otimes v,\\
&[\omega, \sigma\otimes v]=0,
\end{align*}
where $W(3,0)$ acts on $\Omega^2(3)$ via Lie derivative.
Finally, we identify $W_3$ with $\Omega^2(3)^{-1}$ and $\Omega^0(3)$ with $\Omega^3(3)^{-1}$. Besides, we identify $\Omega^2(3)^{-\frac{1}{2}}$ with $W_3^{\frac{1}{2}}$ and we denote by $X_{\omega}$ the vector field corresponding to the $2$-form
$\omega$ under this identification. Then, for $\omega_1, \omega_2\in \Omega^2(3)^{-1/2}$, $u_1, u_2\in\C^2$, we have:
\begin{align*}
&[\omega_1\otimes u_1, \omega_2\otimes u_2]=(X_{\omega_1}(\omega_2)-(\diver(X_{\omega_2}))\omega_1)u_1\wedge u_2,\\
&[f_1\otimes u_1, f_2\otimes u_2]=df_1\wedge df_2\otimes u_1\wedge u_2,\\
&[f_1\otimes u_1, \omega_1\otimes u_2]=(f_1\omega_1+df_1\wedge d(\diver(X_{\omega_1})))\otimes u_1\wedge u_2-
\frac{1}{2}(f_1d\omega_1-\omega_1df_1)\otimes u_1\cdot u_2,
\end{align*}
where, as in the description of $E(3,6)$, $u_1\cdot u_2$ denotes an element in the symmetric square of $\C^2$,
i.e., an element in $\mathfrak{sl}_2$, and $u_1\wedge u_2$ an element in the skew-symmetric square of $\C^2$, i.e., a complex number.
Let $\{e_1, e_2\}$ be the standard basis of $\C^2$ and $E,F,H$ the standard basis of $\mathfrak{sl}_2$. We shall often simplify notation
by writing elements of $E(3,8)$ omitting the $\otimes$ sign.

The so-called principal gradation of $E(3,8)$ is defined by:
$$\deg x_i=-\deg \de_{x_i}=2,\,\,\deg E=\deg F=\deg H=0, \,\, \deg dx_i=2, \,\, \deg e_i=-3.$$
It is a gradation with $0$-th graded component spanned by
the elements $x_i\de_{x_j}$, $E$, $F$ and $H$, and isomorphic to $\mathfrak{sl}_3\oplus \mathfrak{sl}_2\oplus\C$. Besides,
$E(3,8)_{-1}=\langle x_ie_1, x_ie_2~|~ i=1,2,3\rangle$,
$E(3,8)_{-2}=\langle \de_{ x_i}~|~ i=1,2,3\rangle$ and $E(3,8)_{-3}=\langle e_1, e_2\rangle$.

In order to define the Lie conformal superalgebra $RE(3,8)$ we let $\mathfrak g=E(3,8)$ and we proceed as in Section \ref{section7}. We consider the following formal distributions in the variables $z_1,z_2,z_3$ with coefficients in $\tilde {\mathfrak g}$: $a_i(\zb)=\delta(\xb-\zb) \de_{x_i}$, $b_i(\zb)=\sum_{j\neq i}\de_{z_j}\delta(\xb-\zb) d_{ij}$, $c(\zb)=\delta(\xb-\zb)\otimes c$, $e_l(\zb)=\delta(\xb-\zb)\otimes e_l$ and $d_{ijl}(\zb)=\delta(\xb-\zb) d_{ij}\otimes e_l$ for all  $i,j=1,2,3$, $c\in \mathfrak{sl}_2$ and $l=1,2$. The formal  Fourier transform of the brackets between these formal distributions lead us to the following definition.
\begin{definition}
	We consider the free $\C[\de_1,\de_2,\de_3]$-module generated by the following elements:
	\begin{itemize}
\item[-] $a_1, a_2, a_3$  even of degree $-2$, 
\item[-] $b_1, b_2, b_3$ even of degree 2, 
\item[-] $E, F, H$ even of degree  0, which are identified with the standard generators of  $\mathfrak{sl}_2$,
\item[-] $e_1, e_2$ odd of degree $-3$, which are identified with the canonical basis of $\C^2$,
\item[-] $d_{hkl}$  with  $h,k=1,2,3$, $l=1,2$, odd of degree 1,
\end{itemize}
and define $RE(3,8)$ as its quotient by the following relations:
\begin{itemize}
\item $\de_1 b_1+\de_2 b_2+\de_3b_3=0$,
\item $d_{hkl}=-d_{khl}$,
\end{itemize}
with
the following $\la$-brackets:
\begin{align*}
&[{a_i}_{\la}a_j]=-(\lambda_i+\de_i)a_j-\lambda_ja_i+\frac{1}{2}\lambda_i(\lambda_j+\de_j)
\sum_k\lambda_kb_k
\\
&[{a_i}_{\la}c]=-(\de_i+\lambda_i)c\\
&[{a_i}_{\la}b_k]=-(\de_i+\lambda_i)b_k+\delta_{ki}\sum_r\lambda_rb_r\\
&[c_\la b_j]=0\\
&[{b_i}_\la b_k]=0\\
&[{e_i}_\la e_j]=(j-i)\sum_{h,k}\varepsilon_{hk} \lambda_k\de_h a_{t_{hk}}\\
&[{e_i}_\la d_{jkh}]=\varepsilon_{jk}((h-i)a_{t_{jk}}+(i-h)(\lambda_{t_{jk}}+\de_{t_{jk}})
\sum_{r}\lambda_rb_r+\frac{1}{2}(2\lambda_{t_{jk}}+\de_{t_{jk}})e_i\cdot e_h)\\
&[{d_{jkh}}_\la d_{jsl}]=\varepsilon_{jk}(l-h)(1-\delta_{ks})b_j\\
&[{a_i}_\la e_j]=-(\frac{3}{2}\lambda_i+\de_i)e_j-\frac{1}{2}\sum_{k,r}\lambda_k\lambda_i
\de_rd_{krj}\\
&[c_\la e_j]=ce_j-\sum_{i,h}\lambda_i\de_hd_{ih}\otimes ce_j\\
&[c_\la d_{ijh}]=[c, d_{ijh}]\\
&[{b_j}_\la e_k]=-\sum_i\lambda_id_{ijk}\\
&[{b_i}_\la d_{jkl}]=0\\
&[c_\la c']=[c,c']+\sum_k \tr(cc')\lambda_k b_k .\\
&[{a_i}_\la d_{hkl}]=-\big(\frac{3}{2}\lambda_i+\de_i\big)d_{hkl}+ \delta_{ih} \sum_j \lambda_j d_{jkl}-\delta_{ik} \sum_r \lambda_r d_{rhl},
\end{align*}
for every $c,c'\in\mathfrak{sl}_2$.
\end{definition}
\begin{proposition}
	$RE(3,8)$ is a Lie conformal superalgebra of type $(3,0)$ which satisfies 
	Assumptions \ref{assumption}.
\end{proposition}
\begin{proof}
It is enough to check that the $\la$-brackets are consistent with the relations and, as in the case of $E(3,6)$, that the conformal Jacobi identity holds. This follows from a long and tedious but elementary computation and also from the fact that we constructed $RE(3,8)$ as a space of formal distributions with coefficients in the Lie superalgebra $\tilde {\mathfrak g}$. Assumptions \ref{assumption} are immediate.
\end{proof}
\begin{theorem}
 The annihilation algebra ${\mathfrak g}(RE(3,8))$ is isomorphic, as $\Z$-graded Lie superalgebra,  to $E(3,8)$ with principal gradation via the following map:
$$\begin{array}{lcl}
a_i\yb^M & \mapsto & \xb^M \de_{x_i}\\
b_k \yb^M & \mapsto & -\sum_{r\neq k} (\de_{x_{r}}\xb ^{M}) d_{rk}\\
c\yb^M & \mapsto & \xb^M \otimes c\\
 d_{jkl}\yb ^M & \mapsto &  \xb^M d_{jk}\otimes e_l\\
 e_l\yb ^M & \mapsto & \xb^M \otimes e_l\end{array}$$
for all $c\in\mathfrak{sl}_2$, $i,j,k=1,2,3$, $l=1,2$ .
 \end{theorem}
\begin{proof}
It is immediate to check that the map is well-defined. The proof that it is a linear isomorphism is similar to the case of $E(5,10)$. The fact that it is a Lie superalgebra homomorphism
 follows from the following relations which can be verified by direct computations:
 
 \medskip
\begin{enumerate}
\item $[a_i\yb^M,a_j\yb^N]=n_ia_j \yb^{M+N-e_i}-m_ja_i \yb^{M+N-e_j}-\frac{1}{2}\sum_k b_k (\de_{y_i}\de_{y_k}\yb^M)(\de_{y_j}\yb^N)$
\item[(1b)]$[\xb^M \de_{x_i},\xb^N \de_{x_j}]=n_i \xb^{{M+N-e_i}}\de_{x_j}-m_j \xb^{M+N-e_j}\de_{x_i}+\frac{1}{2}\sum_{k,r} (\de_{x_i}\de_{x_k}\xb^M)(\de_{x_r}\de_{x_j}\xb^N)d_{rk}$
\medskip
\item $[a_i\yb^M, c \yb ^N]=n_ic \yb^{M+N-e_i}$
\item[(2b)]$[\xb^M \de_{x_i},\xb^N c]=n_i\xb^{M+N-e_i}c$
\bigskip
\item $[a_i\yb^M, b_k \yb^N]=n_ib_k\yb^{M+N-e_i}+\delta_{ki}\sum_r m_r b_r \yb^{M+N-e_r}$
\item[(3b)]$[\xb^M \de_{x_i},-\sum_{r\neq k}n_r\xb^{N-e_r}d_{rk}]
=-n_i\sum_r (\de_{x_r}\xb^{M+N-e_i})d_{rk}-\delta_{ki} \sum_{r,s} m_r (\de_{x_s}\xb^{M+N-e_r})d_{sr}$\bigskip
\item $[a_i\yb^M, e_j \yb^N ]=(n_i-\frac{1}{2}m_i)e_j \yb^{M+N-e_i}+\frac{1}{2}\sum_{k,r}d_{krj}(\de_{y_k}\de_{y_i}\yb^M)(\de_{y_r}\yb^N)$
\item[(4b)]$[\xb^M \de_{x_i},\xb^N e_j]=((n_i-\frac{1}{2}m_i)\xb^{M+N-e_i}+
\frac{1}{2}\sum_{h,k}\de_{x_h}\de_{x_i}(\xb^M)\de_{x_k}(\xb^N)d_{hk})e_j$\bigskip
\item $[a_i\yb^M, d_{hkl} \yb^N]=(n_i-\frac{1}{2}m_i)d_{hkl}\yb^{M+N-e_i}+\delta_{ih}\sum_j m_j d_{jkl}\yb^{M+N-e_j}-\delta_{ik} \sum_r m_r d_{rhl} \yb^{M+N-e_r}$
\item[(5b)] $[\xb^M \de_{x_i},\xb^N d_{hk}\otimes e_l]=
(n_i-\frac{1}{2}m_i)\xb^{M+N-e_i}d_{hkl}+\delta_{ih}\sum_jm_j\xb^{M+N-e_j}d_{jkl}
-\delta_{ik}\sum_rm_r\xb^{M+N-e_r}d_{rhl}$\bigskip
\item $[c \yb^M, c' \yb^N]=[c,c']\yb^{M+N}+\tr(cc')\sum_k m_k b_k \yb^{M+N-e_k}$ 
\item[(6b)] $[\xb^M c, \xb^N c']=\xb^{M+N}[c,c']+
\tr(cc')\sum_{k,r}m_kn_r\xb^{M+N-e_k-e_r}d_{kr}$\bigskip
\item $[c \yb^M, b_k \yb^N]=0$
\item [(7b)]$[\xb^M c, -\sum_{r\neq k}\xb^{N-e_r}d_{rk}]=0$;\bigskip
\item $[c \yb^M, e_j \yb^N]=ce_j \yb^{M+N}+\sum_{i,h}m_i n_h d_{ih}\otimes ce_j \yb^{M+N-e_i-e_h}$
\item [(8b)]$[\xb^M c, \xb^N e_j]=(\xb^{M+N}+\sum_{k,r}m_kn_r\xb^{M+N-e_k-e_r}d_{kr})\otimes ce_j$\bigskip
\item $[c \yb^M, d_{jkl} \yb^N]=[c,d_{jkl}]\yb^{M+N}$
\item [(9b)]$[\xb^M c, \xb^N d_{jk}\otimes e_l]=\xb^{M+N}d_{jk}\otimes ce_l$\bigskip

\item $[b_k\yb^M ,  b_h \yb ^N ]=0$
\item [(10b)] $[-\sum_{r\neq k}m_r\xb^{M-e_r}d_{rk},   -\sum_{s\neq h}n_s\xb^{N-e_s}d_{rh}]=0$\bigskip
\item $[b_k\yb^M ,   e_j \yb^N]=-\sum_i m_i d_{ikj} \yb^{M+N-e_i}$
\item[(11b)]$[-\sum_{r\neq k}m_r\xb^{M-e_r}d_{rk},\xb^N e_j]=
\sum_{r\neq k}m_r\xb^{M+N-e_r}d_{krj}$\bigskip
\item $[b_k\yb^M ,   d_{jkl} \yb^N]=0$
\item[(12b)] $[-\sum_{r\neq k}m_r\xb^{M-e_r}d_{rk},\xb^N d_{jk}\otimes e_l]=0$\bigskip
\item $[e_i\yb^M,e_j \yb^N]=(j-i)\sum_{k,r}\varepsilon_{kr} m_k n_r a_{t_{kr}}\yb^{M+N-e_k-e_r}$
\item[(13b)]
$[\xb^Me_i,\xb^N e_j]=(j-i) \sum_{k,r}\varepsilon_{kr} m_k n_r \xb^{M+N-e_k-e_r}\de_{t_{kr}}$
\bigskip
\item  $[e_i\yb^M,d_{jkh}\yb^N] = \varepsilon_{jk}\Big((h-i)\big(a_{t_{jk}}\yb^{M+N}+\sum_r b_r (\de_{y_r}\yb^M)(\de_{y_{t_{jk}}}\yb^N)\big)$\\
${}$\hspace{3cm}$+\frac{1}{2}(m_{t_{jk}}-n_{t_{jk}}) \yb^{M+N-e_{t_{jk}}}e_i\cdot e_h  \Big)$
\item[(14b)] $[\xb^Me_i,\xb^N d_{jk}\otimes e_h]=\varepsilon_{jk}\Big(
(h-i)(\xb^{M+N}\de_{t_{jk}}+\sum_{r,s}m_r\xb^{M-e_r}
\de_{x_s}\de_{x_{t_{jk}}}(\xb^N)d_{rs})$\\
${}$\hspace{3.5cm}$+\frac{1}{2}(m_{t_{jk}}-n_{t_{jk}}) \xb^{M+N-e_{t_{jk}}}e_i\cdot e_h\Big)$\bigskip
\item $[d_{jkh}\yb^M,d_{jsl}\yb^N]=\varepsilon_{jk}(l-h)(1-\delta_{ks})b_j\yb^{M+N} $
\item[(15b)]$[\xb^M d_{jk}\otimes e_h, \xb^N d_{js}\otimes e_l]=
(l-h)(\varepsilon_{jk}((n_{t_{jk}}-\frac{1}{2}m_{t_{jk}})\xb^{M+N-e_{t_{jk}}}d_{js}
-\delta_{t_{jk} s}\sum_{r}m_r\xb^{M+N-e_r}d_{rj})$\\
${}$\hspace{4cm}$-\varepsilon_{js}n_{t_{js}}\xb^{N+M-e_{t_{js}}}d_{jk})$.
\end{enumerate}
\end{proof}
\begin{remark} Let $R=RE(3,8)$. As in the case of $E(3,6)$, $E(3,8)_0\cong \mathfrak{sl}_3\oplus\mathfrak{sl}_2\oplus\C Y$ where we can choose 
$Y=\frac{2}{3}\sum_ix_i\de_i$ (see \cite{KR}). Note that $\ad Y$ acts on 
$\mathfrak g(R)_j$ as multiplication by $j/3$ and since $\dim \mathfrak g(R)_{-1}=6$,
$\dim \mathfrak g(R)_{-2}=3$ and $\dim \mathfrak g(R)_{-3}=2$, we have
$\str(\ad Y_{|\mathfrak g(R)_{<0}})=2$. This is consistent with the results of \cite{KR} on degenerate $E(3,8)$-modules.
\end{remark}


\begin{thebibliography}{9}

\bibitem{B} {\sc L.\ Bagnoli,}
{\em Finite irreducible modules over the Lie conformal superalgebra $K'_4$,}
{in preparation (2019).}

\bibitem{BDAK} {\sc B.\ Bakalov, A.\ D'Andrea, V.G.\ Kac,} {\em Theory of finite pseudoalgebras,}
 Advances in Mathematics {\bf 162} (2001), 1--140.

	\bibitem{BKL} {\sc C.\ Boyallian, V.\ G.\ Kac, J.\ I.\ Liberati,}
	{\em Irreducible modules over finite simple Lie conformal superalgebras of type $K$,}
	{J.\ Math.\ Phys.\ {\bf 51} (2010), 1--37.}
	
	\bibitem{BKL1} {\sc C.\ Boyallian, V.\ G.\ Kac, J.\ I.\ Liberati,}
	{\em Classification of finite irreducible modules over the Lie conformal superalgebra $CK_6$,}
	{Comm.\ Math.\ Phys.\ {\bf 317} (2013), 503--546.}
	
	
	\bibitem{BKLR} {\sc C.\ Boyallian, V.\ G.\ Kac, J.\ I.\ Liberati, A.\ Rudakov}
	{\em Representations of simple finite Lie conformal superalgebras of type $W$ and $S$,}
	{J.\ Math.\ Phys.\ {\bf 47} (2006), 1--25.}
	
	\bibitem{CC} {\sc N.\ Cantarini, F.\ Caselli,} {\em Low degree morphisms of $E(5,10)$-generalized Verma modules}, {Algebras and Representation Theory (2019) 10.1007/s10468-019-09925-0, arXiv \tt 1903.11438.}

	\bibitem{CCK}
	{\sc
		S.-J.\ Cheng, N.\ Cantarini, V.\ G.\ Kac,}
	\newblock{\em Errata to Structure of Some $\Z$-graded Lie Superalgebras of Vector Fields,}
	\newblock Transf. Groups\
	{\bf 9} (2004), 399--400.

	
	\bibitem{CaK} {\sc N.\ Cantarini, V.\ G.\ Kac,} {\em Infinite dimensional primitive linearly compact Lie superalgebras}, {Adv.\ Math.\ {\bf 207} (2006), 328--419.}
	
	
	\bibitem{CK}
	{\sc
		S.-J.\ Cheng, V.\ G.\ Kac,}
	\newblock{\em A new $N=6$ superconformal algebra,}
	\newblock Comm.\ Math.\ Phys.\
	{\bf 186} (1997), 219--231.
	
	\bibitem{CK1}
	{\sc
		S.-J.\ Cheng, V.\ G.\ Kac,}
	\newblock{\em Conformal modules,}
	\newblock Asian J.\ Math.\
	{\bf 1} (1997), 181--193.
		
		\bibitem{CK2}
	{\sc
		S.-J.\ Cheng, V.\ G.\ Kac,}
	\newblock{\em Structure of some $\Z$-graded Lie superalgebras of vector fields,}
	\newblock Transf.\ Groups,
	{\bf 4} (1999), 219--272.	

		\bibitem{CL}
	{\sc
		S.-J.\ Cheng, N.\ Lam,}
	\newblock{\em Finite conformal modules over the $N=2,3,4$ superconformal algebras,}
	\newblock J.\ Math.\ Phys.\
	{\bf 42} (2001), 906--933.	

\bibitem{FK}
	{\sc D.\ Fattori, V.\ G.\ Kac}
	\newblock{\em Classification of finite simple Lie conformal
superalgebras,}
	\newblock J.\ Algebra {\bf 258} (2002), 23--59.
		
	\bibitem{K}
	{\sc V.\ G.\ Kac}
	\newblock{\em Classification of infinite-dimensional simple linearly 
		compact Lie superalgebras,}
	\newblock Adv.\ Math.\ {\bf 139} (1998), 1--55.
	
	\bibitem{K1} {\sc V.\ G.\ Kac, } {\em Vertex Algebras for Beginners,} University Lecture Series 10, 2nd Ed., AMS, Providence, RI, 1998.
	
	\bibitem{KR1}
	{\sc V.\ G.\ Kac, A.\ Rudakov}
	\newblock{\em Representations of the exceptional Lie superalgebra $E(3,6)$. I. Degeneracy condition,}
	\newblock Transformation Groups {\bf 7} (2002), 67--86.
	
	\bibitem{KR2}
	{\sc V.\ G.\ Kac, A.\ Rudakov}
	\newblock{\em Representations of the exceptional Lie superalgebra $E(3,6)$. II. Four series of degenerate modules,}
	\newblock Comm.\ Math.\ Phys.\  {\bf 222} (2001), 611--661.
	
	\bibitem{KR}
	{\sc V.\ G.\ Kac, A.\ Rudakov}
	\newblock{\em Complexes of modules over the exceptional Lie superalgebras $E(3,8)$ and $E(5,10)$,}
	\newblock Int.\ Math.\ Res.\ Not.\  {\bf 19} (2002), 1007--1025.
	
	\bibitem{KR3}
	{\sc V.\ G.\ Kac, A.\ Rudakov}
	\newblock{\em Representations of the exceptional Lie superalgebra $E(3,6)$. III. Classification of singular vectors,}
	\newblock J.\ Algebra Appl.\ {\bf 4} (2005), 15--57.

	\bibitem{R0}
	{\sc A.\ Rudakov}
	\newblock{\em Irreducible representations of infinite-dimensional Lie algebras of Cartan type,}
	\newblock Izv.\ Akad.\ Nauk SSSR Ser.\ Mat.\ {\bf 38} (1974), 835--866; Math.\ USSR-Izv.\ 
	{\bf 8} (1974), 836-866.

	
	\bibitem{R}
	{\sc A.\ Rudakov}
	\newblock{\em Morphisms of Verma modules over exceptional Lie superalgebra $E(5,10)$,}
	\newblock {\tt} arXiv 1003.1369v1 (2010), 1--12.
	
	\end{thebibliography}
\end{document}